\documentclass{aart}

\usepackage{amsmath} 
\usepackage{amssymb}
\usepackage{amsthm}
\usepackage{amsxtra}
\usepackage{amscd}
\usepackage{mathabx}
\usepackage{bbm}
\usepackage{mathrsfs}
\usepackage{bm}
\usepackage{xspace} 
\usepackage{amsfonts}
\usepackage{latexsym}
\usepackage{cite} 
\usepackage{graphicx} 
\usepackage{caption} 
\usepackage{subcaption} 
\usepackage{comment} 
\usepackage{esint}

\def\bfseries{\fontseries \bfdefault \selectfont \boldmath}
\def\bfdefault{b}

\usepackage{color}

\numberwithin{equation}{section}

\newcommand{\eq}[1]{\begin{equation}#1\end{equation}}

\newcommand{\al}[1]{\begin{align}#1\end{align}}

\renewcommand{\cal}{\mathcal} 

\newcommand{\ii}{\mathrm{i}} 
\newcommand{\dd}{\mathrm{d}} 
\newcommand{\deq}{\mathrel{\mathop:}=} 
\renewcommand{\leq}{\leqslant} 
\renewcommand{\geq}{\geqslant} 
\renewcommand{\le}{\leq}
\renewcommand{\ge}{\geq}

\renewcommand{\P}{\mathbb{P}}
\newcommand{\E}{\mathbb{E}}
\newcommand{\C}{\mathbb{C}}

\newcommand{\R}{\mathbb{R}}
\newcommand{\T}{\mathbb{T}}

\newcommand{\N}{\mathbb{N}}

\newcommand{\hH}{\hat{H}}
\newcommand{\fg}{\mathfrak{g}}
\newcommand{\hg}{\hat{g}}
\newcommand{\hS}{\hat{S}}
\newcommand{\La}{\Lambda}

\newcommand{\de}{\delta}
\newcommand{\e}{{\varepsilon}}
\newcommand{\ga}{{\gamma}}

\newcommand{\la}{\lambda}

\newcommand{\p}[1]{({#1})}

\newcommand{\pa}[1]{\left({#1}\right)}

\newcommand{\qa}[1]{\left[{#1}\right]}

\newcommand{\h}[1]{\{{#1}\}}

\newcommand{\ha}[1]{\left\{{#1}\right\}}

\newcommand{\absa}[1]{\left\lvert #1 \right\rvert}

\newcommand{\norma}[1]{\left\lVert #1 \right\rVert}

\DeclareMathOperator{\tr}{Tr}

\DeclareMathOperator{\re}{Re}
\DeclareMathOperator{\im}{Im}
\renewcommand{\Im}{\im} 

\newcommand{\Lad}{\Lambda_{\textsc{d}}}
\newcommand{\Lao}{\Lambda_{\textsc{o}}}

\theoremstyle{plain} 
\newtheorem{theorem}{Theorem}[section]
\newtheorem*{theorem*}{Theorem}
\newtheorem{lemma}[theorem]{Lemma}
\newtheorem*{lemma*}{Lemma}
\newtheorem{corollary}[theorem]{Corollary}
\newtheorem*{corollary*}{Corollary}

\newtheorem*{proposition*}{Proposition}
\newtheorem{definition}[theorem]{Definition}
\newtheorem*{definition*}{Definition}

\newtheorem*{example*}{Example}
\theoremstyle{remark} 
\newtheorem{remark}[theorem]{Remark}
\newtheorem*{remark*}{Remark}

\newtheorem*{remarks*}{Remarks}
\newtheorem{notation}[theorem]{Notation}
\newtheorem*{notation*}{Notation}

\newtheorem*{assumption*}{Assumption}
\renewcommand{\qedsymbol}{$\blacksquare$} 

\newcommand{\be}{\begin{equation}}
\newcommand{\ee}{\end{equation}}

\title{Spectral Statistics of Sparse Random Graphs\\with a General Degree Distribution}

\date{\today}

\author{
Ben Adlam\qquad Ziliang Che
\\\\\\
Harvard University\\
Cambridge MA 02138, USA \\ adlam@fas.harvard.edu\qquad zche@math.harvard.edu \\}

\begin{document} 


\maketitle

\vspace{0.7cm}

\begin{abstract}
	We consider the adjacency matrices of sparse random graphs from the Chung-Lu model, where edges are added independently between the $N$ vertices with varying probabilities $p_{ij}$.  The rank of the matrix $(p_{ij})$ is some fixed positive integer. We prove that the distribution of eigenvalues is given by the solution of a functional self-consistent equation.  We prove a local law down to the optimal scale and prove bulk universality. The results are parallel to \cite{Erdos2013b} and \cite{Landon2015}.
\end{abstract}

\vspace{1cm}
\medskip

\newpage

\section{Introduction} \label{sec: introduction}

Since the introduction of the Erd\H{o}s-R\'{e}nyi model \cite{Erdos1959, Erdos1960}, random graphs have been studied extensively as central objects in probability theory \cite{Bollobas1998, Chung2006a}, computer science \cite{Goldreich2008, Vadhan2012,Dembo2014}, and biology \cite{Lieberman2005, Ohtsuki2006, Antal2006, Adlam2014}. In addition to furthering our understanding in these fields, random graphs arise naturally as models that reproduce many properties found in real-world networks \cite{Goffman1969, Wasserman1994, Watts1998, Chakrabarti1999, Albert1999, Barabasi1999, Chung2006a}. In many random graph models, the degrees of the vertices concentrate around a single value with small fluctuations. However, in some real-world networks the behavior of the degrees is evidently different \cite{Aiello2001}. A graph is called scale-free if the number of vertices with degree $j$ is proportional to $j^{-\beta}$ for some positive $\beta$. 

The spectral properties of random graphs are of great interest in applications as they relate to many combinatorial properties of the graph, to the mixing times of Markov chains, \textit{etc} \cite{Aldous2002, Bollobas1998}. In this paper, when we refer to any spectral property of a graph, we mean the eigenvalues and eigenvectors of its adjacency matrix. Many different statistics are used to study the spectra of random graphs. The statistics can be divided into two main types: bulk and edge. Bulk statistics involve the eigenvalues (and their respective eigenvectors) in the interior of the spectrum. More precisely, if we consider the empirical spectral distribution (\textsc{esd}) of an $N\times N$ random matrix, where a point mass of weight $1/N$ is placed at the location of each eigenvalue, then the bulk of the spectrum contains any interval where the limit of the \textsc{esd} has a density bounded away from zero. Edge statistics concern the behavior of the extremal eigenvalues at the edges of the spectrum. At an edge, the limiting density of the \textsc{esd} falls off to zero.

When the \textsc{esd} converges in distribution, we call the limit a global law as this tells us that the density of eigenvalues in a macroscopic interval of length order one for a large finite matrix is well approximated by the global law. At microscopic scales, we can ask about the density of eigenvalues in intervals of length order $N^{\e-1}$ for positive $\epsilon$. A local law proves that the density of eigenvalue in these microscopic windows is also well approximated by the limiting distribution. We cannot expect this to hold for intervals of length less than order $N^{-1}$, as the number of eigenvalues in such an interval does not grow with $N$, and so $N^{\e-1}$ is optimal.To study this convergence of the density of the eigenvalues, the Stieltjes transformation, which is defined as $m_{\mu}(z)\deq \int(t-z)^{-1} \mu(\dd t)$ for a measure $\mu$, is often used as this is equivalent to convergence in the \textsc{esd}. The parameter $z=E+\ii \eta\in\C^{+}$ tracks the density of the eigenvalues at energy $E$ for spectral windows of length $\eta$. Sometimes the limiting distribution has a closed form, like the semicircle distribution, but for more complicated models, it is often specified as a functional self-consistent equation of the Stieltjes transformation. The bounds on the rate of convergence of the Stieltjes transformation further divide local laws. A strong law is optimal and provides the bound $(N\eta)^{-1}$, whereas weaker laws bound the convergence with a smaller power of $N\eta$.

While the global and local laws are model dependent, the fluctuations of the eigenvalues about the locations predicted by the global law show substantial universality. In particular, the behavior of the gaps between eigenvalues and the $n$-point correlation function of the eigenvalues from many real-symmetric random matrix ensembles are often the same as the Gaussian orthogonal ensemble. This observation is formalized by the Wigner-Dyson-Gaudin-Mehta conjecture, or bulk universality conjecture, that states that the local statistics of Wigner matrices are universal in the sense that they depend only on the symmetry class of the matrix, but are otherwise independent of the details of the distribution of the matrix entries. This conjecture for all symmetry classes has been established in a series of papers \cite{Erdos2009b, Erdos2012a, Erdos2011, Erdos2012b, Erdos2013a, Erdos2012c} After this work began, parallel results were obtained in certain cases in \cite{Tao2011a, Tao2010}.

Edge statistics often exhibit universality too. The extremal eigenvalues are studied as a point process \cite{Anderson2010} and typical results involve the almost sure location that the largest eigenvalue converges to and the fluctuations about this location. Often these fluctuations, which are independent of the details of the model and depend only on the symmetry class, fall into the Tracy-Widom universality class. 

Returning to random graphs, for the original Erd\H{o}s-R\'{e}nyi model, which is distinguished from Wigner matrices by its sparsity and discreetness, much has been proved. The degree distribution for the Erd\H{o}s-R\'{e}nyi model concentrates about $Np$ with small fluctuations (so long as $p$ is not too small) and the global distribution is given by the Wigner semicircle. A strong local law at the optimal scale for the density of the eigenvalues is known, so long as $pN\gg (\log N)$, along with universality of both the bulk and edge statistics, when $pN\gg N^{2/3}$ \cite{Erdos2012a, Erdos2013b}. Recently, universality was proved to hold down to the scale $pN\gg N^\e$ in \cite{Landon2015}.  

However, when the degree distribution is given by a power-law, the global distribution of the eigenvalues does not follow the Wigner semicircle---instead it follows a power-law \cite{Rodgers2006}. This can be seen empirically in real-world data and in various models that produce scale-free random graphs. To state more precise results, we must focus on a model: in the Chung-Lu random graph model \cite{Chung2002, Chung2003}, the expected degree sequence of the graph $(d_{1},\ldots,d_{N})$ on $N$ vertices is specified and an undirected edge is added between vertices $i$ and $j$ independently with probability $p_{ij}\deq d_{i}d_{j}/D$, where $D\deq \sum_{k}d_{k}$. Many different degree distributions can be produced by the Chung-Lu model, including power-laws with an exponent $\beta>3$. For general degree sequences, the almost sure location of the largest eigenvalue is known \cite{Chung2003} and the global law of the spectrum is obtained as the solution of a functional equation for its Stieltjes transform (see \eqref{eq: sce general}) , which was derived in \cite{Rodgers2006} using the replica method. 

The distinguishing feature of the Chung-Lu model is that the variance matrix has low-rank. Motivated by this (see Section \ref{sec: examples}), we study the following more general model, where $H=(h_{ij})$ is a real symmetric $N\times N$ random matrix. The entries of $H$ are independent (up to the symmetry constraint), have mean zero, and variance matrix $S=(s_{ij})$ with low-rank. The matrix $H$ can also be sparse and we parameterize the sparseness with $q=N^{\kappa}$ for $\kappa\in (0,1]$, which we point out is different from the parameter in \cite{Erdos2012a, Erdos2013b}. The moments of the entries decay as $\E \absa{h_{ij}}^{k}\le s_{ij}N^{-1}q^{1-k/2}$ for $k\geq 3$. Just like the Chung-Lu model, the global law of the spectrum is obtained as the solution of a functional equation for its Stieltjes transform (see Section~\ref{sec: solution}). We derive this self-consistent system in Section \ref{sec: proof of local law}. In Section \ref{sec: solution}, we prove existence and uniqueness of the solution using the Brouwer fixed-point theorem and an elementary argument. Additionally, we prove the regularity and stability of the solution and provide its asymptotics. 

For this ensemble, we prove a local law in the bulk at the optimal scale $\eta\gg N^{-1+\e}$. The result is parallel to that in \cite{Erdos2013b}.  The local law enables us to prove bulk universality, using a recent result in \cite{Landon2015a}.  Our bulk universality result is parallel to that in \cite{Landon2015}, and is proved using the same technique.  We point out that in the non-sparse case, that is $q=N$, our results are contained in the results by Ajanki, Erd\H{o}s, and Kr\H{u}ger \cite{Ajanki2015}.  The novelties in our paper are: we introduce sparsity into the ensemble; we analyze a different equation \eqref{eq: sce general} than the quadratic vector equation (\textsc{qve}) introduced in \cite{Ajanki2015a}, in order to handle possible singularities in $s_{ij}$. 

In the appendix, we give a new proof to the existence and uniqueness of solution the \textsc{qve}.  We also study a general sparse ensemble, where we no longer require $s_{ij}$ to be low-ranked, but instead assume it to be flat in the sense that $s_{ij}\leq C/N$.  We state and sketch the proof of a local law and bulk universality of this ensemble.  As mentioned above, Ajanki, Erd\H{o}s and Kr\H{u}ger \cite{Ajanki2015} have earlier proved the local law and bulk universality in the non-sparse setting ($q=N$).  

The layout of the paper is as follows: In Section \ref{sec: results}, we define the random-sign model, introduce some basic definitions, and precisely state our main results. In Section \ref{sec: examples}, we introduce some specific ways to produce degree sequences from a distribution $\pi$ that satisfy our assumptions. In Section \ref{sec: proof of local law}, we prove Theorem \ref{thm: local law}, the local law. Section \ref{sec: solution} is devoted to proving the existence, uniqueness, and regularity of the solution to \eqref{eq: sce general}.  In Section \ref{sec: universality}, we use a Green's function comparison argument and a result of \cite{Landon2015a} to prove the universality in the bulk for $q \gg N^{\e}$. In Section \ref{sec: 01 model}, we consider a different model, where the entries only take the values $0$ and $1$.  We sketch the proof for the same local law and bulk universality for this model under the assumption that $q\ll N^{1/2}$.  Appendix \ref{sec: general sce} is devoted to proving the existence, uniqueness, and regularity of the solution to a general self-consistent system. In Appendix \ref{sec: general local law}, we state the local law and bulk universality for a general matrix ensemble whose limiting density is given by the self-consistent equations in Section \ref{sec: general sce}.

\section{Definition and main results} \label{sec: results}

In the Chung-Lu random graph model \cite{Chung2002, Chung2003}, the expected degree sequence of the graph $(d_{1},\ldots,d_{N})$ on $N$ vertices is specified and an undirected edge is added between vertices $i$ and $j$ independently with probability
\eq{\label{eq: def of pij}
	p_{ij}\deq \frac{d_{i}d_{j}}{D}\,,
}
where $D\deq \sum_{k}d_{k}$. Note that the probabilities $p_{ij}$ are only defined if $\max_{i}d^{2}_{i}\leq D$. Obviously, the exact degree of a vertex $i$ is random, but its expected degree is exactly $d_{i}$. Under a particular choice of the expected degree sequence, which we specify in Section \ref{sec: examples}, the model can produce a scale-free graph for any $\beta>3$. Moreover, the original Erd\H{o}s-R\'{e}nyi model is recovered as a special case when the probabilities $p_{ij}$ do not depend on $i$ or $j$. In fact, under some mild assumptions on the expected degree sequence, a wide class of degree distributions can be produced by this model. Often the most interesting choices for $p_{ij}$ lead to sparse graphs.

To study the spectral properties of the adjacency matrices, denoted by $A_{N}$, of random graphs from the Chung-Lu model, we must center and rescale their entries; this produces a symmetric $N\times N$ matrix model $H_{N}\equiv \pa{h_{ij}}_{i,j=1}^{N}$. The spectral properties cannot be studied if the graph is too sparse. We parameterize the sparseness with $q=N^{\kappa}\deq D/N$ for $\kappa\in(0,1]$, which we point out is different from the parameter in \cite{Erdos2012a, Erdos2013b}---in fact, $q_{\textsc{here}}=q_{\textsc{there}}^{2}$. The parameter $q$ is the order of the average degree of the graph and note that we require that $\kappa$ is positive. There are two ways to transform the entries to have expectation zero. The entries can be centered by subtracting their mean. Alternatively, each entry can be multiplied by an independent random sign (up to the symmetry constraint). Next, to ensure the density of the eigenvalues of $A_{N}$ converges to a distribution of order one, we rescaled by $q^{-1/2}$.

Centering and rescaling with the transformation $h_{ij}=(a_{ij}-p_{ij})/\sqrt{q}$ leads to independent entries (up to the symmetry constraint) that have the following distribution
\eq{\label{eq: without random signs}
	\P\qa{ h_{ij} = \pa{ 1-p_{ij}}/\sqrt{q}\; }   = p_{ij}  \quad\text{and}\quad  \P\qa{ h_{ij}=-  p_{ij}/\sqrt{q}\; } = 1-p_{ij} \,.
}
The random sign approach with the transformation $h_{ij}\deq s_{ij}a_{ij}/\sqrt{q}$ (where $s_{ij}$ are independent random signs) again leads to independent entries, but they have the distribution
\eq{
	\P\qa{ h_{ij} = \pm 1/\sqrt{q}\; } = p_{ij}/2 \quad\text{and}\quad  \P\qa{ h_{ij}=0 } = 1-p_{ij}\,.
}
The first matrix model has the variance matrix
\eq{
	S=\pa{ \frac{d_{i}d_{j}\pa{D-d_{i}d_{j}} }{qD^{2} }},
}
which has rank two, and the second has the variance matrix
\eq{
	S=\pa{ \frac{d_{i}d_{j}}{qD} },
}
which has rank one.

Thus, we are lead to the following generalized ensemble, which includes the above matrix models as examples. Denote $\N_N\deq \{1,\dots, N\}$. We introduce an ensemble of $N\times N$ symmetric random matrices ${H^{(N)}=\p{h_{ij}^{(N)}}_{i,j\in\N_{N}}}$. The parameter $N$ is large and we often omit explicitly indicating the dependence of various quantities on $N$. As discussed in Section \ref{sec: introduction}, we motivate the ensemble by considering the adjacency matrices of random graphs from the Chung-Lu model, where edges are added independently between two of the $N$ vertices, $i$ and $j$, with varying probabilities $p_{ij}$ (see Equation \eqref{eq: def of pij}).

Let $s_{ij}=\E\absa{h_{ij}}^{2}$.  We assume that 
\eq{\label{eq: moment assumptions}
	\E h_{ij}=0\quad\text{and}\quad\E \absa{h_{ij}}^k \leq \frac{s_{ij}}{N q^{k/2-1}}
}
for fixed $k\in\N$. We assume that the variance matrix $S=(s_{ij})_{i,j=1}^N$ is of low-rank and has the form 
\eq{\label{eq: s}
	s_{ij}=\frac{1}{N}\sum_{k=1}^r \gamma^{(k)}_i \gamma^{(k)}_j \, ,
}
where $r$ is a fixed positive integer and $(\gamma^{(k)}_i)_{k\in \N_r,i\in \N_N}$ satisfy the following assumptions:
	\begin{enumerate}
		\item $\gamma^{(k)}_i\geq 1$ for all $k\in \N_r$ and $i\in \N_N$;
		\item there is an $M>0$ (not dependent on $N$) such that $\frac{1}{N} \sum _{k,i} (\gamma^{(k)}_i)^2 \leq M\,.$
	\end{enumerate}
We also require that 
\eq{
	qs_{ij}\leq 1\,,
}
so that the matrix can be realized through the adjacency matrix of a random graph.
\begin{remark}
It is possible that $\max_{i,j}s_{ij}\rightarrow \infty$ as $N\rightarrow \infty$, as shown in Subsection \ref{sec: power law}.  In the non-sparse case $q=N$, this is impossible, since $qs_{ij}\leq 1$ implies $s_{ij}\leq 1/N$.
\end{remark}
We denote 
\eq{\label{def: theta}
	\theta_i\deq \sum_k^r \gamma_i^{(k)}\,.
}
\begin{definition}
	If $H=(h_{ij})_{i,j\in\N_{N}}$ satisfies the conditions above, we call $H$ a generalized Chung-Lu ensemble.
\end{definition}

We are going to analyze the behavior of the resolvent of $H$, that is,
\eq{\label{def: resolvent}
	G(z)=(H-z)^{-1}
}
for $z=E+\ii\eta\in\C^+$.

We define the Stieltjes transform of a measure $\rho$ as $m_{\rho}(z)\deq\int_\R \frac{\rho(\dd x)}{x-z}$. Denote the eigenvalues of $H$, ordered in increasing size, by $\pa{\la_{1},\ldots, \la_{N}}$. We are interested in the asymptotic spectral statistics of $H$, so we consider the empirical spectral distribution of $H$, defined as $\mu_{N} \deq {\frac{1}{N}\sum_i\delta_{\lambda_i}}$, and its Stieltjes transform
\eq{
	m_N(z) = \int_\R \frac{ \mu_N (\dd x)}{x-z} =\frac{1}{N} \sum_i \frac{1}{\lambda_I-z}\,.
}
In terms of the resolvent \eqref{def: resolvent}, $m_N$ can be written
\eq{
	m_N(z)= \frac{1}{N}\sum_ i^N G_{ii}\,.
}
It turns out that (see Section \ref{sec: proof of local law} for a derivation), the limiting behavior of $m_{N}$ is given by the limiting behavior of the unique solution of the following system:
\eq{\label{eq: sce general}
	\begin{split}
		m(z) &= -\frac{1}{N}\sum_i^N \frac{1}{z+\sum_{j=1}^r\gamma^{(j)}_i u^{(j)}(z) }\,\\
		u^{(k)}(z) &= -\frac{1}{N}\sum_i^N \frac{\gamma^{(k)}_i}{z+\sum_{j=1}^r \gamma^{(j)}_i u^{(j)}(z) }
	\end{split}
}
for $k\in\N_{r}$.  

We denote 
\eq{
	g_i=- \frac{1}{z+\sum_{j=1}^r\gamma^{(j)}_i u^{(j)}(z) }\,.
}

Thus, to show the self-consistent system \eqref{eq: sce general} gives the correct limiting behavior, our goal is to estimate the family of quantities defined below.
\begin{notation}\label{not: parameters}
Define the $z$-dependent quantities
\eq{
	\Lad\deq \max_i  \theta_i \absa{{G_{ii}-g_i }} \,,
}
\eq{
	\Lao\deq \max_{i\neq j}\absa{\sqrt{\theta_i\theta_j}G_{ij}}\,,
}
and 
\eq{\label{definition of Lambda}
	\Lambda\deq\max\{\Lad,\Lao\}\,.
}
We also set the control parameter
\eq{\label{def: Phi}
	\Phi\deq \frac{1}{\sqrt{q}}+\frac{1}{\sqrt{N\eta}}\,.
}
\end{notation}

The existence and uniqueness of a solution to \eqref{eq: sce general} will be proved in Section \ref{sec: solution}.  Throughtout this paper, we denote $(u^{(k)})_{k\in\N_{r}}$ and $m$ to be the solution to $\eqref{eq: sce general}$ without specification.  Our main theorem concerns the asymptotic behavior of eigenvalues of $H$ in the bulk of the spectrum, that is, where the asymptotic distribution of the spectrum has a positive density.  To make this precise, we introduce the following definition.

 \begin{definition}[Bulk interval] \label{def: bulk interval}
Assume that $H$ is a generalized Chung-Lu ensemble.  Let $(u^{(k)})_{k\in\N_{r}}$ and $m$ to be the solution to $\eqref{eq: sce general}$.  We call a bounded interval $I\subset \R$ a \emph{bulk interval} if there is a positive $c_I$ which does not depend on $N$ such that $\im m(z)>c_I$ on ${\{E+\ii\eta:E \in I, 0<\eta \leq 1\}}$. 
\end{definition}

Such an interval might not exists for an arbitrary generalized Chung-Lu ensemble, but it does exist in many cases of interest, for example, when $(\gamma^{(k)}_i)_{k\in \N_r,i\in \N_N}$ has a limit distribution, see Section \ref{sec: examples}.

Before stating our main theorem concerning the matrix model $H$, we need to introduce some notation. We follow the conventions of \cite{Erdos2013a}.

\begin{definition}[Stochastic domination]\label{def: stochastic domination}
Let 
\eq{
	X=\pa{X^{(N)}(u): N\in \N, \, u\in U^{(N)}} \quad\text{and}\quad Y=\pa{Y^{(N)}(u): N\in \N, \, u\in U^{(N)}}
}
be two families of random variables with $Y$ nonnegative. We say that $X$ is stochastically dominated by $Y$, uniformly in $u$, if for all positive constants $\e$ and $D$, we have 
\eq{
	\sup_{u\in U^{(N)} }\P [\absa{X(u)}>N^\e Y(u)] \leq N^{-D}
}
for large enough $N$.  Moreover, we denote this by $X\prec Y$.
\end{definition} 
Now we state our main theorem below. We point out that Theorem 1.6 of \cite{Ajanki2015} contains the non-sparse case of our theorem, that is when the sparse parameter $q=N$.  In fact, their theorem translated to this special case is much more general, since it not only deals with the bulk, but the edges as well.
\begin{theorem}[The local law]\label{thm: local law}
Let $I$ be a bulk interval (as in Defintion \ref{def: bulk interval}). Fix positive $\delta $ and define an $N$-dependent spectral domain 
\eq{
	\cal{D}_\delta^I\deq \ha{ E+\ii\eta: E\in I, N^{\delta-1}\leq \eta \leq10}\,.
}
Then, on the domain $\cal{D}_\delta^I$, we have $\Lambda\prec \Phi$.
\end{theorem}

As an application of the local law, we have bulk universality in terms of the n-point correlation functions of eigenvalues.  The non-sparse $q=N$ was done in Theorem 1.15 of \cite{Ajanki2015}.

Let $\rho^{(n)}$ be the n-point correlation functions of the eigenvalues of $H$ and $\rho^{(N)}$ be the density on $I$. We denote $\rho_{GOE}^{(n)}$ to be the n-point correlation function of eigenvalues of a $GOE$ and $\rho_{sc}$ to be the density of the semicircle law:
\be
	\rho_{sc}(E)= \frac{\sqrt{[4-E^2]_+}}{2\pi}\,.
\ee

\begin{theorem}[Bulk universality]\label{thm: bulk universality}
Let $O\in C^\infty_0(\R^n)$ be a test function.  Let $I$ be a bulk interval. Fix a parameter $b=N^{c-1}$ for arbitrarily small $c$. We have,
\eq{\begin{split}
	\lim_{N\rightarrow \infty} \int_{E-b}^{E+b} \int_{\R^n} O(\alpha_1,\dots,\alpha_n) \left[ \frac{1}{\rho(E)^n}\rho^{(n)} \left( E'+\frac{\alpha_1}{N\rho(E)},\dots,E'+\frac{\alpha_n}{N\rho(E)}\right)\right.\\\left.
	-\frac{1}{(\rho_{sc}(E))^n} \rho_{GOE}^{(n)} \left( E''+\frac{\alpha_1}{\rho_{sc}(E)},\dots,E''+\frac{\alpha_n}{\rho_{sc}(E)}\right)\right]\dd \alpha_1\dots\dd\alpha_n \frac{\dd E'}{2b} =0
	\end{split}
}
for any $E''\in (-2,2)$.
\end{theorem}

\section{Examples of interest}\label{sec: examples}

In this section, we give specific examples of models that motivate the ensemble we defined in Section \ref{sec: results}. In particular, we consider the cases where ${(\gamma^{(k)}_i)_{k\in \N_r,i\in \N_N}}$ converges to a limit (in a sense to be made precise) and when the random matrix ensemble is given by the adjacency matrix of a random graph. We show that these models satisfy the assumptions of our theorems; note that Theorems \ref{thm: local law} and \ref{thm: bulk universality} require that there exists a bulk interval, which is any interval where the spectral density has a lower bound independent of $N$.   

\subsection{When $(\gamma^{(k)}_i)$ has a limit} \label{subsec: limit} For $k\in\N_{r}$, we define
\eq{
	f_k^{(N)}(x)\deq \sum_ i ^N 1_{[\frac{i-1}{N},\frac{i}{N})} \gamma^{(k)}_i\,.
}
If $f_k^N\rightarrow f_k$ in some sense, we expect the following equations from \eqref{eq: sce general}:
\eq{\label{eq: limit sce}
	\begin{split}
	u_k &=-\int_0^1 \frac{f_k(x)\dd x}{z+\sum_{k=1}^r f_k(x)u_k}\\
	m_0&=-\int_0^1 \frac{\dd x}{z+\sum_{k=1}^r f_k(x)u_k}
	\end{split}\,,
}
where
\eq{
	\inf_{x\in[0,1]}f_k(x)\geq 1\quad\text{and} \int_0^1 f_k(x)^2 \dd x \leq M\,.
}
We note that this system of equations for $k=1$ has been derived before in \cite{Rodgers2006} using the replica method. For \eqref{eq: limit sce}, which is the limiting, integral version of \eqref{eq: sce general}, we can prove uniqueness, existence, and holomorphicity of the solution. 
\begin{theorem}\label{thm: limit solution}
For any $z\in \C^+$, there exists a unique $(u_1, \dots, u_r, m_0)$ satisfying equations \eqref{eq: limit sce}. In addition, the function $(u_1, \dots, u_r, m_0)$ is holomorphic in $z$.
\end{theorem}

\begin{proof}
	The proof is a trivial modification of that of Theorem \ref{thm: solution}.
\end{proof}

\begin{theorem}
	If $f_k^N \rightarrow f$ in $L^2([0,1])$, then any bounded closed interval $I\subset \R$ on which $\im m_0>0$ is a bulk interval.
\end{theorem}
\begin{proof}
	This theorem is a consequence of the stability of the system \eqref{eq: limit sce}.  See Theorem \ref{thm: stability}.
\end{proof}

\subsection{Power law with $\beta>3$}\label{sec: power law}
For simplicity, we assume $r=1$, that is, the matrix $s_{ij}$ has rank one.  Assume that 
\eq{
	\gamma_i= \left(\frac{i}{N}\right)^{-\mu},
}
where $0<\mu<1/2$, then the sequence of functions
\eq{
	f^N(x)= \sum_ i ^N 1_{[\frac{i-1}{N},\frac{i}{N}]} \gamma_i\,
}
converges to $f(x)=x^{-\mu}$ in $L^2([0,1])$.  From the previous subsection, we know that bulk intervals exist.  Therefore, Theorem \ref{thm: local law} and Theorem \ref{thm: bulk universality}  apply to this case.  

In particular, this model gives a random graph that has a power law as the degree distribution with exponent $\beta=1+\frac{1}{\mu}>3$.  To be precise, we consider a random graph with $N$ vertices and the probability of the $i$-th vertex connects to the $j$-th vertex is $qs_{ij}=N^{\kappa-1}  \left(\frac{i}{N}\right)^{-\mu} \left(\frac{j}{N}\right)^{-\mu}$.  Assume that different edges are independent.  Thus, the adjacency matrix of this graph is a random matrix with variance $qs_{ij}$.  Now the expected degree of the $i$-th edge is $\sum_j qs_{ij}$, proportional to $\left(\frac{i}{N}\right)^{-\mu}$. Thus the number of edges that has degree $[Nx, N(x+\dd x)]$ is asymptotically $Nx^{-1-\frac{1}{\mu}}$, up to a normalizing constant, which is exactly a power law distribution with exponent $\beta=1+\frac{1}{\mu}>3$.

\section{Proof of the local law}\label{sec: proof of local law}

In this section, we prove Theorem \ref{thm: local law}.  
\subsection{Some notation and resolvent identities}
We introduce some notations which will be useful later on.
\begin{definition} [Minors]
For $\T \subset \{1,\dots, N\}$, we define $H^{(\T)}$ by
\eq{
	\pa{H^{(\T)}}_{ij}\deq h_{ij} 1_{i\notin \T} 1_{j\notin \T}\,.
}
Also, the corresponding Green's functions are defined as
\eq{
	G_{ij}^{(\T)}(z)\deq (H^{(\T)}-z)^{-1}_{ij}\,.
}
In a similar way, we use the notation
\eq{
	\sum_{i}^{(\T)} \cdot \deq \sum_{i\in N\setminus \T}\cdot
}
Moreover, we abbreviate $(\h{i})$ by $(i)$ and $(\T\cup\h{i})$ by $(\T i)$.
\end{definition}
\begin{definition}[Partial expectation]
Let $X=X(H)$ be a random variable. We define $Q_i$ by
\eq{
	Q_i(X)\deq X-\E\qa{X \mid H^{(i)}}\,.
}
\end{definition}
The following lemma is frequently used in the proof of the local law.  See Lemma 4.5 in \cite{Erdos 2013a} for reference.
\begin{lemma}[Resolvent identities] \label{lem: resolvent}
For any Hermitian matrix $H$ and $\T\subset\{1,\dots,N\}$, the following identities hold. If $i,j,k\notin \T$ and $i,j\neq k$, then
\eq{\label{eq: off diagonal 1}
	G_{ij}^{(\T)}=G_{ij}^{(\T k)}+\frac{G_{ik}^{(\T)}G_{kj}^{(\T)}}{G_{kk}^{(\T)}}\quad  \text{and} \quad \frac{1}{G_{ii}^{(\T)}}=  \frac{1}{G_{ii}^{(\T k)}}- \frac{G_{ik}^{(\T)}G_{ki}^{(\T)}}{G_{ii}^{(\T)}G_{ii}^{(\T k)}G_{kk}^{(\T)}}\,.
}
If $i,j\notin \T $ satisfy $i\neq j$, then
\eq{\label{eq: off diagonal 2}
	G_{ij}^{(\T)}=-G_{ii}^{(\T)}\sum_k^{(\T i)} h_{ik} G_{kj}^{(\T i)} =-G_{jj}^{(\T)}\sum_k^{(\T j)}  G_{ik}^{(\T j)}  h_{kj}\,.
}
If $i\not\in \T$, then
\eq{\label{eq: schur}
	\frac{1}{G_{ii}^{(\T)} }= h_{ii}-z-\sum_{j,k}^{(\T i)}h_{ij} G_{jk}^{(\T i)}h_{ki}\,.
}
\end{lemma}

\subsection{Derivation of the limit equation}
For the sake of simplicity, we only consider the case where $S$ has rank 2, that is,
\eq{
	s_{ij}=\frac{1}{N}(\alpha_i\alpha_j +\beta_i\beta_j)\,.
}
Higher ranked cases can be handled similarly.  We can also recover the rank-one case by setting $\alpha_i=\beta_i$, for all $i$. Using the Schur complement formula (see Equation \eqref{eq: schur} of Lemma \ref{lem: resolvent}), we have
\eq{\label{eq: application of Schur}
	\frac{1}{G_{ii}}=-z-\sum_k s_{ik} G_{kk} +R_i\,,
}
where the error term $R_i$ is defined as
\eq{\label{eq: error}
	R_i \deq h_{ii}+\sum_k s_{ik} \frac{G_{ik}G_{ki}}{G_{ii}}-Q_i \sum_{k,l}^{(i)}h_{ik}G_{kl}^{(i)}h_{li} \,.
}

Define
\eq{
	U\deq\frac{1}{N}\sum_k \alpha_k G_{kk}\quad\text{and}\quad V\deq\frac{1}{N}\sum_k \beta_k G_{kk}\,.
}
After rearranging \eqref{eq: application of Schur}, we get the following equations
\eq{\label{eq: sce with error}
	\begin{split}
	U&=-\frac{1}{N}\sum_k\frac{\alpha_k}{z+\alpha_k U+\beta_k V-R_i}\,,\\
	V&=-\frac{1}{N}\sum_k\frac{\beta_k}{z+\alpha_k U+\beta_k V-R_i}\,,\\
	m_N&=-\frac{1}{N}\sum_k\frac{1}{z+\alpha_k U+\beta_k V-R_i}
	\end{split}
}

Formally, we neglect the error terms and replace $G_{ii}$ with $g_i$, then Equation \eqref{eq: application of Schur} becomes
\eq{
	\frac{1}{g_i}=-z-\sum_k s_{ik} g_k\,.
}
\eqref{eq: sce with error} becomes
\eq{\label{eq: sce}
	\begin{split}
		u(z) &= -\frac{1}{N}\sum_{k=1}^{N} \frac{\alpha_{k}}{z+\alpha_{k}u(z)+\beta_{k}v(z)} \\
		v(z) &= -\frac{1}{N}\sum_{k=1}^{N} \frac{\beta_{k}}{z+\alpha_{k}u(z)+\beta_{k}v(z)} \\
		m(z) &= -\frac{1}{N}\sum_{k=1}^{N}\frac{1}{z+\alpha_{k}u(z)+\beta_{k}v(z)}
	\end{split} \,.
}
In Section \ref{sec: solution}, we will prove that this system has a unique solution and the solution is stable under small perturbation.

\subsection{Large deviation estimates} 
To control the error terms in the self-consistent system, we need several large deviation estimates, which we state here. We omit the proof here, which is a slight modification of Lemma A.1 in \cite{Erdos2013b}.  
\begin{lemma}\label{lem: LDE} 
Assume that the family of random variables $\pa{h_{ij}}_{i,j\in\N_{N}}$ satisfy Equation \eqref{eq: moment assumptions}.

\begin{enumerate}
\item
Fix an index $i\in \N_N$. Let $\p{A_{1},\ldots,A_{N}}$ be a family of random variables that is independent of $\p{ h_{i1},\ldots,h_{iN} }$ and satisfies $A_j\prec 1$ for $j\in \N_N$, then we have
\eq{
	\sum_j (h_{ij}(t)^2-s_{ij}) A_j  \prec \sqrt{\frac{\theta_i}{q}}\,.
}
\item Fix distinct indices $i,j\in\N_N$. Let $(B_{kl})_{k,l=1}^{N}$ be a family of random variables that is independent of $\p{ h_{i1},\ldots,h_{iN} }$ and $\p{h_{1j},\ldots,h_{Nj}}$ and that satisfies $B_{kl}\prec 1$ for $k,l\in \N_N$, then we have
\eq{
	\sum_{k\neq l} h_{ik}(t) B_{kl} h_{lj}(t) \prec \sqrt{\frac{\theta_i\theta_j}{q}}+\left( \frac{\theta_i \theta_j}{N^2} \sum_{k,l} \theta_k\theta_l \absa{B_{kl}}^2\right)^{1/2}\,.
}
\end{enumerate}
\end{lemma}

\subsection{Proof of the local law}
Our strategy is to first assume that $\Lambda$ is bounded by a large control parameter $N^{-c}$ on some event that has high probability. With this \textit{a priori} bound on $\Lambda$, we prove that $\Lambda$ is actually bounded by a much smaller control parameter $\Phi$ (see Notation \ref{not: parameters}) with very high probability.  Thus, the probability that $\Lambda$ lies in between $N^{-c}$ and $\Phi$ is very small.  However, when $\eta$ is large (order one), we can easily show that $\Lambda$ is bounded by $\Phi$. Finally, we use a continuity argument to push this estimate all the way down to $\eta\geq \frac{N^\delta}{N}$ for arbitrarily small positive $\delta$.  

The following lemma essentially says that with very high probability, $\Lambda$ is outside the interval $[\Phi,N^{-c}]$.  So $[\Phi,N^{-c}]$ is sometimes called the ``forbidden'' region.
\begin{lemma}\label{lem: forbidden}
Assume that $I$ is a bulk interval (see Definition \ref{def: bulk interval}).  Let $\phi$ be the indicator function of some (possibly z-dependent) event. For any spectral domain $\cal{D}\subset \cal{D}_I$, if $\phi \Lambda\prec M^{-c}$ for some positive $c$, then $\phi\Lambda \prec \Phi$.
\end{lemma}

\begin{proof}
On the domain $\cal{D}_I$, it is easy to see that $u$ and $v$ are bounded and $\im u$ and $\im v$ are bounded below.
By the definition of $\Lambda$ and $\cal{D}_{0}^{I}$, we see that 
\eq{\label{eq: phi Gii}
	\phi G_{ii}\asymp \frac{\phi}{\theta_i}\,,
}
where $\theta_i$ is defined as in \eqref{def: theta}:
\eq{
	\theta_i:=\alpha_i+\beta_i\,.
}
First, we bound $\Lao$. We need the identity
\eq{\label{eq: lambdao}
	G_{ij}=-G_{ii}G_{jj}^{(i)}\pa{h_{ij}-\sum_{k,l}^{(ij)} h_{ik} G_{kl}^{(ij)} h_{lj} }\,,
}
which is easily obtained by iterating Equation \eqref{eq: off diagonal 2} of Lemma \ref{lem: resolvent} . We get an \textit{a priori} bound on $G_{jj}^{(i)}$ using Equation \eqref{eq: off diagonal 1} of Lemma \ref{lem: resolvent}:
\eq{\label{eq: Gjj minor}
	\phi G_{jj}^{(i)}\prec \phi G_{jj} +\phi \frac{\theta_k}{\sqrt{\theta_j \theta_k}\sqrt{\theta_k\theta_j}}M^{-c}\prec \frac{1}{\theta_j}\,.
}
Similarly, one can prove that 
\eq{
	\phi G_{kl}^{(ij)}\prec \frac{1}{\sqrt{\theta_k\theta_l}},\quad \phi G_{kk}^{(ij)}\prec \frac{1}{\theta_k}\,.
}
By (ii) of Lemma \ref{lem: LDE}, \eqref{eq: lambdao} yields
\eq{\label{gij bound}
		\phi \absa{G_{ij}}\prec \frac{1}{\theta_i\theta_j}\pa{\sqrt{\frac{\theta_i\theta_j}{q}}+\sqrt{ \frac{\theta_i\theta_j}{N^2} \sum_{k,l} \theta_k\theta_l \absa{G_{kl}^{(ij)}}^2} }\,.
}
By the Cauchy-Schwartz inequality and Ward's identity, we have 
\eq{\label{application of Ward}
	\phi \sum_{k,l} \theta_k\theta_l \absa{G_{kl}^{(ij)}}^2\leq \phi \sum_k\frac{\theta_k^2\Im G_{kk}^{(ij)}}{\eta} \prec \frac{N}{\eta}\,,
}
where we used the bounds $\phi G_{kk}^{(ij)}\prec \frac{1}{\theta_k}$ and $\sum_k \theta_k \prec N$.  Therefore, Equation \eqref{gij bound} is stochastically dominated by
\eq{
	\frac{1}{\theta_i\theta_j}\pa{\sqrt{\frac{\theta_i\theta_j}{q}}+\sqrt{ \frac{\theta_i\theta_j}{N\eta}} }=\frac{1}{\sqrt{\theta_i\theta_j}}\Phi\,.
}
So $\Lao\prec \Phi$.

Next, we bound the error term $R_i$.  First, $h_{ii}\prec 1/\sqrt{q}\prec\Phi$. Second, using Equation \eqref{eq: phi Gii} and the bound $\Lao\prec \Phi$, we see
\eq{\label{eq: second term in error}
	\phi\sum_k s_{ik} \frac{G_{ik}G_{ki}}{G_{ii}}=\phi\sum_k^{(i)} s_{ik} \frac{G_{ik}G_{ki}}{G_{ii}}+ \phi s_{ii} G_{ii}\prec \frac{\theta_i^2}{N}\sum_k \frac{1}{\theta_i} \Phi ^2+\frac{1}{q} \prec \theta_i \Phi\,.
}
Third, we use Lemma \ref{lem: LDE} and the Ward identity (as we did in \eqref{application of Ward}) to estimate
\al{
	\begin{split}
		\phi Q_i \sum_{k,l}^{(i)}h_{ik}G_{kl}^{(i)}h_{li} &=\phi\sum_k^{(i)} (h_{ik}^2-s_{ik}) G_{kk}^{(i)} +\phi\sum_{k\neq l}^{(i)}h_{ik}G_{kl}^{(i)}h_{li} \\
		&\prec \frac{\theta_i}{\sqrt{q}}+\frac{\theta_i}{\sqrt{N\eta}}\sqrt{\frac{1}{N}\sum_k \theta_k^2} \prec \theta_i \Phi\,.
	\end{split}
}
Therefore,	 $\phi R_i \prec \theta_i\Phi$, which yields
\eq{
	\phi\theta_i\left(G_{ii}+\frac{1}{z+\alpha_i U+\beta_i V-R_i}\right)\prec \Phi\,.
}

By the stability of the self consistent equation, Theorem \ref{thm: stability}, we have $\phi\theta_i\absa{G_{ii} -g_i}\prec \Phi$ and therefore, $\Lambda\prec\Phi$.
\end{proof}

For $\eta$ large, we establish an \emph{a priori} bound on $\Lambda$, which we use as an input for Lemma \ref{lem: forbidden}.

\begin{lemma}\label{lem: initial estimate}
For any bounded interval $I\subset\R$, we have $\Lambda\prec \Phi$ on the line segment $\{E+\ii\eta: E\in I,\eta=10\}$.
\end{lemma}

\begin{proof} 
We proceed as in Lemma \ref{lem: forbidden}. First, we want to show that $G_{ii}\prec \frac{1}{\theta_i}$. In Lemma \ref{lem: forbidden}, this estimate was obtained using the \textit{a priori} assumption $\Lambda\prec N^{-c}$, but now we only have the trivial bound $\absa{G_{ii}}\leq \frac{1}{\eta}=\cal{O}(1)$. First, note that 
\eq{
	\absa{ \sum_{k\neq l}^{(i)} h_{ik}G_{kl}^{(i)} h_{li} } \prec \frac{C}{\sqrt{q}}\,.
}
Next, we proceed with Equation \eqref{eq: schur}, to see
\eq{\label{gii initial}
	\begin{split}
		\absa{G_{ii}} &=\absa{\frac{1}{-z-\sum_k^{(i)} h_{ik}^2 G_{kk}^{(i)} -\sum_{k\neq l}^{(i)} h_{ik}G_{kl}^{(i)} h_{li} }}\\
		&\leq \frac{1}{ \eta+\sum_k^{(i)} h_{ik}^2 \im G_{kk}^{(i)}-\cal{O}_{\prec}\pa{\frac{1}{\sqrt{q}}}}\,.
	\end{split}
}
Next, we get a lower bound for $\im G_{kk}^{(i)}$. Using Lemma \ref{lem: LDE} and Equation \eqref{eq: phi Gii}, we see
\eq{
	\im G_{kk}^{(i)} \geq \frac{\eta+\sum_l^{(k)} h_{kl}^2\im G_{ll}^{(ik) }+{\cal{O}_\prec \left( \frac{1}{\sqrt{q}}\right)}}{\absa{-z-\sum_{k, l}^{(ik)} h_{kl}G_{lm}^{(ik)} h_{mk}}^2}\geq \frac{\eta+\cal{O}_\prec \left( \frac{1}{\sqrt{q}}\right)}{\cal{O}_\prec (\theta_i^2)}\,.
}
Thus, $\frac{\eta}{\theta_i^2}\prec \im G_{kk}^{(i)}$. Note that the same holds for $G_{kk}$.  Now, we return to Equation \eqref{gii initial}, to get $\absa{G_{ii}} \prec\frac{1}{\theta_i}$.

Once we have this bound, we can move on to estimate $G_{ij}\prec \Phi/\sqrt{\theta_i\theta_j} $ as we did in \eqref{gij bound}.
The estimate for $R_i$ proceeds in exactly the same way as in Lemma \ref{lem: forbidden}, we omit the details: $R_i\prec \theta_i\Phi$. Thus,
\eq{
	\absa{G_{ii}+\frac{1}{z+\alpha_i U+\beta_i V-R_i}}\ga_i\prec \frac{R_i}{\theta_i^2}\prec \frac{\Phi}{\theta_i}\,,
}
By Theorem \ref{thm: stability initial}, this implies $\theta_i\absa{G_{ii}-g_i}\prec \Phi$, which in turn implies $\Lad\prec \Phi$. Therefore, $\Lambda\prec\Phi$.
\end{proof}

Now, we complete the proof of Theorem \ref{thm: local law} with a continuity argument.

\begin{proof}[Proof of theorem \ref{thm: local law}]

We choose a lattice $\Delta\subset D_\delta^I$ such that $\absa{\Delta}\leq N^{20}$ and the $N^{-5}$-neighborhood of $\Delta$ covers the whole $D_\delta^I$. Take $\phi=[\Lambda\leq N^{-c}]$ for some positive and small enough constant $c$.  Lemma \ref{lem: forbidden} states that for any positive, large $D$, $\e$ small (such that $\Phi N^\e<N^{-c}$), and $N$ large enough,
\eq{
	\P[ \exists w\in\Delta : \Phi N^\e \leq \Lambda \leq N^{-c}] \leq N^{-D}\,.
}
Since $\Lambda$ is Lipschitz on $D_\delta^I$ with constant at most $N^2$, we have
\eq{
	\P[ \exists z\in\Delta : 2\Phi N^\e \leq \Lambda \leq N^{-c}] \leq N^{-D}\,.
}
Using Lemma \ref{lem: initial estimate} and the continuity of $\Lambda$, we have 
\eq{
	\P[ \exists z\in\Delta : \Lambda\geq \Phi N^\e] \leq N^{-D}\,.
}
Thus, $\Lambda \prec \Phi$.
\end{proof}

\section{Solution to the self-consistent equation}\label{sec: solution}

In this section, we analyze the self-consistent system \eqref{eq: sce}. We prove the solution exists and is unique, we prove the equation is stable, and we prove some asymptotic behavior of the solution. 

For the sake of simplicity, we only consider the two dimensional case, as higher dimensional cases can be handled similarly. Thus, we have the equations
\eq{\label{eq: sce}
	\begin{split}
		u(z) &= -\frac{1}{N}\sum_{k=1}^{N} \frac{\alpha_{k}}{z+\alpha_{k}u(z)+\beta_{k}v(z)} \\
		v(z) &= -\frac{1}{N}\sum_{k=1}^{N} \frac{\beta_{k}}{z+\alpha_{k}u(z)+\beta_{k}v(z)} \\
		m(z) &= -\frac{1}{N}\sum_{k=1}^{N}\frac{1}{z+\alpha_{k}u(z)+\beta_{k}v(z)}
	\end{split} \,.
}
This is a two-dimensional version of equations \eqref{eq: sce general}.  Next, we have an existence and uniqueness theorem for the solution of equation \eqref{eq: sce}.

\begin{theorem}\label{thm: solution}
	For any $z\in \C^+$ there is unique $(u,v)\in \C^+\times \C^+$ satisfying 	\eqref{eq: sce}.  Moreover, $(u,v)$ is holomorphic in $z$.
\end{theorem}
\begin{remark}
Thus, $m$ is uniquely determined by $(u,v)$, by the third equation of \eqref{eq: sce}.
\end{remark}

\begin{proof}
	Fix $z\in \C^+$. Define $F:\overline{\C^+}\times \overline{\C^+}\rightarrow \C^+\times \C^+$ by 
	\eq{
		(\sigma,\tau)\mapsto \left(\frac{1}{N}\sum_k\frac{\alpha_k}{z+\alpha_k \sigma+\beta_k\tau},\frac{1}{N}\sum_k\frac{\alpha_k}{z+\alpha_k 		\sigma+\beta_k\tau}\right)\,.
	}
	As a rational function, it is continuous.  By Brouwer's fixed point theorem, there is a fixed point ${(u,v)\in\C^+\times \C^+}$ such that $F(u,v)=(u,v)$.
	
	To see uniqueness, we first note that by taking the imaginary part of \eqref{eq: sce} that
	\eq{
		\im u \geq \frac{1}{N}\sum_k \frac{ \alpha _k(\alpha_k \im u+\beta_k \im v)}{\absa{z+\alpha_k u+\beta_k v}^2} \quad\text{and}\quad\im v \geq \frac{1}{N}\sum_k \frac{ \beta_k (\alpha_k \im u+\beta_k \im v)}{\absa{z+\alpha_k u+\beta_k v}^2}\,.
	}
	For $T$ defined by 
	\eq{\label{def: t}
		T:=\frac{1}{N}\sum_k \frac{1}{\absa{z+\alpha_k u+\beta_k v}^2}\begin{bmatrix}
			\alpha_k^2 &\alpha_k\beta_k \\ \alpha_k\beta_k & \beta_k^2
		\end{bmatrix}\,,
	}
	the Perron-Frobenius theorem implies  the spectral radius $r(T)$ is less than or equal to 1. 
	
	Now, assume for the sake of contradiction that there is another solution $(u',v')$, then again we have an analogue of Equation \eqref{def: t} and $r(T')\leq 1$, where $T'$ is defined likewise:
	\eq{\label{def: t prime}
		T':=\frac{1}{N}\sum_k \frac{1}{\absa{z+\alpha_k u+\beta_k v}^2}\begin{bmatrix}
			\alpha_k^2 &\alpha_k\beta_k \\ \alpha_k\beta_k & \beta_k^2
		\end{bmatrix}\,.
	}
	So, we have
	\eq{\label{eq: uniqueness u}
		u-u'= \frac{1}{N}\sum_k \frac{ \alpha _k(\alpha_k (u-u')+\beta_k(v-v'))}{(z+\alpha_k u+\beta_k v)(z+\alpha_k u'+\beta_k v')}
	}
	and
	\eq{\label{eq: uniqueness v}
		v-v'=\frac{1}{N} \sum_k \frac{ \beta_k(\alpha_k(u-u')+\beta_k(v-v'))}{(z+\alpha_k u+\beta_k v)(z+\alpha_k u'+\beta_k v')}\,.
	}
	Denote
	\eq{
		\tilde{T}:=\frac{1}{N}\sum_k \frac{1}{\absa{z+\alpha_k u+\beta_k v}\absa{z+\alpha_k u'+\beta_k v'}}\begin{bmatrix}
			\alpha_k^2 &\alpha_k\beta_k \\ \alpha_k\beta_k & \beta_k^2
		\end{bmatrix}\,.
	}
	and $\xi=(\absa{u-u'}, \absa{v-v'})$. We take the absolute values of equations \eqref{eq: uniqueness u} and \eqref{eq: uniqueness v}, thus
	\eq{\label{eq: abs xi}
		\xi=\tilde{T}\xi\,.
	}
	Then we take the Euclidean norm to see
	\eq{
		\norma{\xi}\leq \norma{\tilde{T}\xi} \,.
	}
	However, applying Cauchy-Shwarz inequality to \eqref{eq: abs xi} yields
	\eq{
		\norma{\tilde{T}\xi}<\frac{1}{2}( \norma{T\xi}+\norma{T'\xi})\leq \norma{\xi}\,,
	}
	which is a contradiction.
\end{proof}

We also have a stability theorem in the bulk.  This theorem says that the solution to equation \eqref{eq: sce} is stable under small perturbation of the equation.

\begin{theorem}\label{thm: stability}
	Let $(u,v)$ be the solution to \eqref{eq: sce}.  Assume $I\in \R$ is a bounded internal and $c>0$ is a constant such that $\im m\geq c $ on 
	\eq{
		\cal{D}_I = \{z: \re z \in I\,, \im z \in(0,10)\}\,.
	}
	Then there are $\e, C\in \R^+$, depending on $c$ but not on $N$, such that the following holds.  For any $z\in \cal{D}_I$, assume that $(u',v')$ satisfies
	\eq{\label{eq: perturbed sce}
		\begin{split}
	u'+\frac{1}{N}\sum_ k \frac{\alpha_k}{z+\alpha_k u'+\beta_k v'}=r_1 \\
	v'+\frac{1}{N}\sum_ k \frac{\beta_k}{z+\alpha_k u'+\beta_k v'}=r_2
		\end{split}\,.
	}
	and $\max\{\absa{r_1},\absa{r_2},\absa{u'-u},\absa{v'-v}\}\leq \e$, then
	\eq{
		\max\{\absa{u'-u},\absa{v'-v}\}\leq C \max\{\absa{r_1},\absa{r_2}\}\,.
	}
\end{theorem}

\begin{proof}
	Denote
	\eq{
		\hat{T}:=\frac{1}{N}\sum_k \frac{1}{(z+\alpha_k u+\beta_k v)(z+\alpha_k u'+\beta_k v')}\begin{bmatrix}
			\alpha_k^2 &\alpha_k\beta_k \\ \alpha_k\beta_k & \beta_k^2
		\end{bmatrix}\,.
	}

	Then, subtracting \eqref{eq: perturbed sce} from \eqref{eq: sce},
	\eq{
		(I-\hat{T})(u-u',v-v')= (r_1,r_2)\,.
	}
	The assumption $\im m\geq c $ implies that $u,u',v,v'$ are bounded and $\im u, \im u', \im v, \im v'$ are bounded below.  Denote	\eq{
		T_1:=\frac{1}{N}\sum_k \frac{1}{(z+\alpha_k u+\beta_k v)^2}\begin{bmatrix}
			\alpha_k^2 &\alpha_k\beta_k \\ \alpha_k\beta_k & \beta_k^2
		\end{bmatrix}\,.
	}
	Then $\norma{T_1-\hat{T}}\leq c_2 \max\{\absa{u'-u},\absa{v'-v}\}$.  Thus we have
	\eq{
		\max\{\absa{u'-u},\absa{v'-v}\} \lesssim \norma{(I-T_1)^{-1}}\left(\max\{\absa{r_1},\absa{r_2}\} -(\max\{\absa{u'-u},\absa{v'-v}\})^2\right)\,.
	}
	It remains to show that the eigenvalues of $T_1$ are bounded away from $1$, so that $(I-\hat{T})^{-1}$ is uniformly bounded. This, however, is guaranteed by Lemma \ref{lem: spectrum}.
\end{proof}

\begin{lemma}\label{lem: spectrum}
	$K\subset \C^+$ is a compact set.  Assume that $A=\begin{bmatrix} a_{11} & a_{12} \\ a_{21} & a_{22} \end{bmatrix}$, where $a_{ij} \in K, i,j\in \{1,2\}$ and that the spectral radius of
	\eq{\label{def: abs matrix}
		\absa{A}:= \begin{bmatrix} \absa{a_{11}} & \absa{a_{12} } \\ \absa{a_{21} } & \absa{a_{22} } \end{bmatrix}
	}
	is $\leq 1$.  Then there is a $\delta>0$ depending only on $K$ such that the eigenvalues of $A$ are bounded away from $1$ by $\delta$.
\end{lemma}

\begin{proof}
	Assume that this is not true, then we can take a sequence $A_n$ satisfying the same conditions as $A$, with $\lambda_n$ and $\xi_n\neq 0$ such that
	\eq{
		A_n\xi_n = \lambda_n \xi_n, \text{ with } \lambda_n \rightarrow 1\,.
	}
We take a subsequence to get $A$ and $\xi$ such that 
	\eq{
		A\xi =\xi\,.
	}
By Perron-Frobenius theorem, we know that $\absa{A}\absa{\xi} =\absa{\xi}$, which means $a_{11}\xi_1$ is parallel to $\xi_1$.  This contradicts with $a_{11}\in K$.
\end{proof}

We also need a stability theorem when $\eta$ is large.  The difference is that we do not assume an \textit{a priori} bound for $\max\{ \absa{r_1},\absa{r_2}\}$ here.

\begin{theorem}\label{thm: stability initial}
	Let $(u,v)$ be the solution to \eqref{eq: sce}.  Assume $I\in \R$ is a bounded internal.  Then there are $\e, C\in \R^+$, depending on $c$, $I$ but not on $N$, such that the following holds.  For any $z\in \{\zeta: \re \zeta \in I, \im \zeta =10\}$, assume that $(u',v')$ satisfies
	\eq{\label{eq: perturbed sce}
		\begin{split}
	u'+\frac{1}{N}\sum_ k \frac{\alpha_k}{z+\alpha_k u'+\beta_k v'}=r_1 \\
	v'+\frac{1}{N}\sum_ k \frac{\beta_k}{z+\alpha_k u'+\beta_k v'}=r_2
		\end{split}.
	}
	and $\max\{\absa{u'-u},\absa{v'-v}\}\leq \e$, then
	\eq{
		\max\{\absa{u'-u},\absa{v'-v}\}\leq C \max\{\absa{r_1},\absa{r_2}\}\,.
	}
\end{theorem}

\begin{proof}
	As in the last theorem, we get
	\eq{
		(I-\hat{T})(u-u',v-v')= (r_1,r_2).
	}
	Recall the definition \eqref{def: t prime} of $T'$. We take the imaginary part of \eqref{eq: perturbed sce} to get
	\eq{
		(I-T') \begin{bmatrix} \im u' \\ \im v' \end{bmatrix} =\begin{bmatrix}\frac{1}{N} \sum \frac{\eta\alpha_k}{\absa{z+\alpha_k u'+\beta_k v'}^2}+\im r_1 \\ \frac{1}{N} \sum \frac{\eta\beta_k}{\absa{z+\alpha_k u'+\beta_k v'}^2}+\im r_1 \end{bmatrix} \,.
	}
	When $r_1$ and $r_2$ are small enough, the right hand side is positive, thus $r(T')\leq 1$ by Perron-Frobenius theorem. We already know $r(T)\leq 1$, thus $r(\hat{\absa{T}})\leq 1$ (see \eqref{def: abs matrix} for the definition of $\hat{\absa{T}}$).  By Lemma \ref{lem: spectrum}, the eigenvalues of $\hat{T}$ are bounded away from $1$, thus $(I-\hat{T})^{-1}$ is uniformly bounded.	
\end{proof}

\section{Bulk universality}\label{sec: universality}
In this section, we prove the universality of $n$-point correlation functions.  The key ingredient is a universality theorem for deformed GOE matrices \cite{Landon2015a}, which we state below.  

A matrix $V$ is said to be $(l,G)$-regular at $E$ if for some parameter $l$ such that
\eq{\
	\frac{1}{N}\leq l \leq N^{-\e_1} \quad\text{and}\quad N^{\e_1} l\leq G^2 \leq N^{-\e_1}
}
there exist positive constants $c_V$ and $C_V$ such that
\eq{
	c_V\leq \im m_V(E'+i\eta)\leq C_V
}
uniformly for $E'\in(E-G,E+G)$ and $l\leq \eta \leq 10$.  In our case, $V$ is $H$ diagonalized and $l=N^{-1+\e}$, so $T=N^{-1+\e'}$ for arbitrarily small $\e'>\e>0$.  Here is their theorem.

\begin{theorem} [B. Landon \& H.-T. Yau \cite{Landon2015a}]\label{optimal speed}
Let 
\eq{
	H=V+\sqrt{T}W\,, 
}
where $V$ is a (deterministic or random) diagonal matrix and $W$ is a standard $GOE$ matrix.  Suppose that $V$ is $(l,G)$-regular at $E$ and that $N^{-\e}G^2\geq T\geq N^\e l$ for some positive $\e$. Let $O\in C^\infty_0(\R^n)$ be a test function. Fix a parameter $b=N^{c-1}$ for any positive $c$ satisfying $c<\e/2$. We have,
\eq{
	\begin{split}
		\lim_{N\rightarrow \infty} \int_{E-b}^{E+b} \int_{\R^n} O(\alpha_1,\dots,\alpha_n) \left[ \frac{1}{(\rho_{fc,T}^{(N)}(E))^n} \rho_T^{(n)} \left( E'+\frac{\alpha_1}{N\rho_{fc,T}^{(N)}(E)},\dots,E'+\frac{\alpha_n}{N\rho_{fc,T}^{(N)}(E)}\right)\right.\\\left.
		-\frac{1}{(\rho_{sc}(E))^n} \rho_{GOE}^{(n)} \left( E''+\frac{\alpha_1}{\rho_{sc}^{(N)}(E)},\dots,E''+\frac{\alpha_n}{\rho_{sc}^{(N)}(E)}\right)\right]\dd \alpha_1\dots\dd\alpha_n \frac{\dd E'}{2b} =0.
	\end{split}
}
\end{theorem}

Our strategy to prove bulk universality is as follows.  Let $H_t=(h_{ij}(t))$ be a symmetric $N\times N$ matrix. The dynamics of the matrix entries are given by the stochastic differential equations
\eq{\label{eq: Ht}
	\dd h_{ij}(t)= \frac{\dd B_{ij}(t)}{\sqrt{N}}- \frac{h_{ij}(t)\dd t}{2 N s_{ij}}\,,
}
where $B$ is symmetric with $(B_{ij})_{1\leq i<j\leq N}$ and $(B_{ii}/\sqrt{2})_{i=1}^{N}$ a family of independent Brownian motions. The initial data $H_0$ is our target matrix $H$. We shall prove a comparison theorem, Theorem \ref{comparison}, which states that, for $T=N^{-1+\e}$ with $\e\ll 1$, the spectral statistics of $H_T$ agree with those of $H_0$. Next, we show that $H_T$ can be written as the sum of two random matrices,
\eq{
	H_T=\hH + \vartheta G\,,
}
where $H_T$ has a local law (Theorem \ref{thm: local law hH}) and $\vartheta G$ is a small GOE matrix. Next, we apply Theorem \ref{optimal speed} to get bulk universality for $H_T$.  Thus Theorem \ref{comparison} says that universality also holds for $H_0$.

We start with the comparison between $H_0$ and $H_T$.  The following lemma is a slight modification of Lemma A.1 in \cite{Erdos2012b}.  The main difference here is that we do not assume $s_{ij}\asymp \frac{1}{N}$.
\begin{lemma}\label{comparison 1}
Assume $s_{ij}$ are as in \eqref{eq: s}. Suppose that $F$ is a smooth function of the matrix elements $(h_{ij})_{i\leq j}$ satisfying
\eq{
	\sup_{0\leq s\leq t, i\leq j, \mathbb{\theta}} \E\left( (Ns_{ij}^{-1} \absa{h_{ij}(s)}^3 + \absa{h_{ij}(s)}) \absa{\partial_{ij}^3 F(\mathbb{\theta} H_s)}\right) \leq M\,,
}
where $(\mathbb{\theta}H)_{ij} =\theta_{ij} h_{ij}$ with $\theta_{kl}=1$ unless $\{k,l\}=\{i,j\}$ and $0\leq \theta_{ij}\leq 1$, then
\eq{
	\E F(H_t) -\E F(H_0) = \cal{O}(tN) M\,.
}
\end{lemma}

\begin{remark}
In our case, $t\leq N^{-1+\e}$, thus we only need $M\leq N^{-c}$ for some $c\geq \e$. 
\end{remark}

In order to apply this lemma to our case, we need some uniform bound on the Green's functions ${G_t=(H-z)^{-1}}$ for $t\in[0,T]$.  Clearly, for every $t\geq 0$, we have $\E (h_{ij}(t)^2)=s_{ij}$. Moreover, $h_{ij}(t)$ is a small perturbation of the initial $h_{ij}(0)$ and has the moment bound
\eq{\label{eq: moment bounds}
	\E \absa{h_{ij}(t)^p} \leq \frac{s_{ij}}{N q^{p/2 -1}}\,.
}  
Thus, we expect the same local law to hold for $H_T$.

\begin{theorem}[The local law for $H_t$]\label{thm: local law Ht}
Assume that $I$ is a bulk interval.  Fix positive $\delta$ and define an $N$-dependent spectral domain $\cal{D}_\delta^I\deq \ha{ E+\ii\eta: E\in I, N^{\delta-1}\leq \eta \leq10}$. Let $\pa{{h_{ij}(t)}}_{i,j=1}^{N}$ be defined as above in \eqref{eq: Ht} and fix $t\in[0,N^{-1+\e}]$.  Therefore, on the domain $\cal{D}_\delta^I$, we have $\Lambda\prec \Phi$.
\end{theorem}

The following comparison lemma is a consequence of Lemma \ref{comparison 1} and Theorem \ref{thm: local law Ht}.  It is a modification of Lemma 5.2 in \cite{Landon2015}. 

\begin{lemma}\label{comparison 2}
Assume that $I$ is a bulk interval. Let $H_t$ be defined as in \eqref{eq: Ht}. Let positive $\delta$ be arbitrary and choose an $\eta$ such that $N^{-1-\delta}\leq \eta \leq N^{-1}$.  For any sequence of complex parameters $z_j = E_j\pm i\eta, j=1,\dots, n$, with $\absa{E_j}\in I$ with an arbitrary choice of the signs, we have the following.  Let $G_t(z)=(H_t-z)^{-1}$ be the resolvent and let $F(x_1, x_2, \dots, x_n)$ be a test function such that for any multi-index $\alpha=(\alpha_1, \dots, \alpha_n)$ with $1\leq\absa{\alpha} \leq 3$ and for any positive, sufficiently small $\omega$, we have 
\eq{
	\max\left\{ \absa{\partial^\alpha F(x_1,\dots, x_n)} : \max_j \absa{x_j} \leq N^\omega \right\} \leq N^{C_0\omega}
}
and
\eq{
	\max\left\{ \absa{\partial^\alpha F(x_1,\dots, x_n)} : \max_j \absa{x_j} \leq N^2 \right\} \leq N^{C_0}
}
for some constant $C_0$.
Therefore, for any $\delta$ with $N^{-1-\delta}\leq \eta \leq N^{-1}$ and for any choices of the signs in the imaginary part of $z_j$, we have
\eq{
	\absa{\E[F\left( \frac{1}{N} \tr G_t(z_1),\dots, \frac{1}{N} \tr G_t (z_n) \right) -\E[F(G_t\rightarrow G_0)]]}\leq \frac{Ct N^{1+c\delta}}{\sqrt{q}}\,,
}
where $c$ and $C$ are constants depending on $C_0$.
\end{lemma}

\begin{proof}
We consider only the $n=1$ case for simplicity.  We want to show that 
\eq{
	\absa{\E[F\left( \frac{1}{N} \tr G_t(z_1) \right) -\E[F\left( \frac{1}{N} \tr G_0(z_1) \right)]]}\leq \frac{Ct N^{1+c\delta}}{\sqrt{q}}\,.
}
We need to compute the derivative of $\frac{1}{N}\tr G$ with respect to the $(i,j)$-th entry.  Note that 
\eq{
	\absa{\partial_{ab}^k \frac{1}{N} \tr G }\leq \frac{1}{N} \sum_k \absa{\tr (G J )^kG}\,,
}
where $J$ is the matrix whose entries all vanish except the $(a,b)$ and $(b,a)$-th entries. The right hand side is a sum of products of off-diagonal entries of $G$ and the number of terms is order $N$.  Thus, we need a bound on the off-diagonal entries of $G$, down to the scale $\eta \geq N^{-1-\delta}$.  In order to get this bound, we first derive a delocalization bound for the eigenvectors in the bulk. For $\lambda_k \in I$, we set $E=\lambda_k$
\eq{
	\im G_{ii}=\sum_\alpha \frac{ \absa{u_\alpha^i}^2\eta}{\absa{\lambda_\alpha-z}^2} \geq \frac{\absa{u_k^i}^2}{\eta}\,.
}
By Theorem \ref{thm: local law Ht}, $\im G_{ii}$ is bounded uniformly in $\cal{D}_\delta^I$, on an event with probability $1-N^{-D}$.  Hence on this event, we have $\absa{u_k^i}^2\leq N^{-1+\delta}$ as long as $\lambda_k\in I$.  This still holds if we replace $I$ with a slightly bigger interval $\tilde{I}$ (independent of $N$), because of the analyticity of $m$ near $I$.  Now we estimate, for $E\in I$, that
\eq{
	\absa{G_{jk}(E+i\eta)} \leq \sum_i \frac{\absa{u_i^j u_i^k}}{\absa{\lambda_i-z}} =\sum_{\lambda_i \in \tilde{I}}\frac{\absa{u_i^j u_i^k}}{\absa{\lambda_i-z}} +\sum_{\lambda_i \notin \tilde{I}}\frac{\absa{u_i^j u_i^k}}{\absa{\lambda_i-z}} \leq CN^{-1+\delta} \sum_{\lambda_i \in \tilde{I} }\frac{1}{\absa{\lambda_i-z}}+C\,.
}
We define a dyadic decomposition of the eigenvalues lying in $\tilde{I}$: let
\eq{
	U_0  \deq\{ j:\lambda_j\in \tilde{I},\absa{\lambda_j-E} \leq N^{-1+\delta}\}\,, 
}
\eq{
	U_n \deq \{j: \lambda_j\in \tilde{I},2^{n-1} N^{-1+\delta} < \absa{\lambda_j -E} \leq 2^n N^{-1+\delta}\}\,, 
}
for $1\leq n \leq \log_2 N$, and
\eq{
	U_{\log_2 N+1}\deq\{ j:\lambda_j\in \tilde{I}, N^\delta < \absa{\lambda_j -E}\}\,.
}
With probability $1-N^{-D}$, we have $\absa{U_n} \leq C2^n N^\delta$, for $0\leq n \leq \log_2 N$,  which can be seen below.  For any $J\in \tilde{I}$ such that $\absa{J}\geq N^{-1+\delta}$, we set $E'$ to be the middle point of $J$ and $\eta=\frac{\absa{J}}{2}$, then we have
\[
	C\geq \im m(E+i\eta)=\frac{1}{N}\sum_i \frac{\eta}{(\lambda_i -E)^2 +\eta^2} \geq \frac{1}{N}\sum_{\lambda_i\in J} \geq \frac{1}{\absa{J}{N}}\#\{j: \lambda_j \in J\}\,.
\]
The dyadic decomposition enables us to estimate
\eq{
	\absa{G_{jk}(E+i\eta)} \leq CN^{-1+\delta} \sum_n \sum_{i\in U_n}\frac{1}{\absa{\lambda_i -E -i\eta}} \leq CN^{3\delta}\,, 
}
for $\eta\ge N^{-1-\delta} $ and therefore, $\frac{1}{N} \absa{\tr \partial_{ab}^k G}\leq CN^{3(k+1)\delta}$ on an event with probability $1-N^{-D}$.  Outside this event, we have a deterministic bound $\frac{1}{N} \absa{\tr \partial_{ab}^k G} \leq C N^{3(1+\delta)}$. Finally, we apply Lemma \ref{comparison 1} to complete the proof. 
\end{proof}

Lemma \ref{comparison 2} readily implies the following comparison theorem for correlation functions:
\begin{theorem}\label{comparison}
Assume that $I$ is a bulk interval. Let $H_t$ be defined as in \eqref{eq: Ht} and $T=N^{-1+\e}$ with $\e$ arbitrarily small. Let $\rho_0^{(n)}$ and $\rho_T^{(n)}$ be the n-point correlation functions of the eigenvalues of $H$ and $H_T$, then for any  test function $O\in C^\infty_0(\R^n)$, we have 
\eq{
	\lim_{N\rightarrow \infty } \int_{\R^n} O(\alpha_1,\dots, \alpha_n) \left[ \rho_0^{(n)} \left( E+\frac{\alpha_1}{N},\dots, E+\frac{\alpha_n}{N}\right)-\rho_T^{(n)} \left( E+\frac{\alpha_1}{N},\dots, E+\frac{\alpha_n}{N}\right)\right]=0\,.
}
\end{theorem}
\begin{proof}
This theorem is a consequence of Theorem 2.1 in \cite{Erdos2012b} and Lemma \ref{comparison 2} .
\end{proof}

It remains to show that $H_T$ has bulk universality.  We note that since $s_{ij} \geq c/N$, $H_T$ has a small Gaussian component, that is, $H_T=\hH+c\sqrt{T}G$.  We state without proof the local law for $\hH$ here:

\begin{theorem}[The local law for $\hH$]\label{thm: local law hH}
Assume that $I\subset \R$ is a bulk interval. Fix positive $\delta $ and define an $N$-dependent spectral domain $\cal{D}_\delta^I\deq \ha{ E+\ii\eta: E\in I, N^{\delta-1}\leq \eta \leq10}$. Let $(\hat{h}_{ij})$ be defined as above.  Therefore, on the domain $\cal{D}_\delta^I$, we have $\Lambda\prec \Phi$.
\end{theorem}

\begin{proof}[Proof of Theorem \ref{thm: bulk universality}]
By Theorem \ref{optimal speed} and Theorem \ref{thm: local law hH}, the conclusion of Theorem \ref{thm: bulk universality} holds for $H_T$ in place of $H$.  Then, by Theorem \ref{comparison}, the same conclusion holds for $H$.
\end{proof}

\section{Adjacency matrix without random signs} \label{sec: 01 model}
In this section, we consider a different matrix ensemble derived from the adjacency matrix of a random graph without random signs (recall Equation \eqref{eq: without random signs}) and outline the proof of a local law and bulk universality. We consider $A=(a_{ij})$ such that
\eq{\label{eq: aij definition 01}
	\P\qa{ a_{ij} = 1 } = \frac{qs_{ij}}{N} \quad\text{and}\quad  \P[ a_{ij}=0 ] = 1-\frac{qs_{ij}}{N}\,,
}
where $s_{ij}$ has the form \eqref{eq: s}. Note that $\E a_{ij}\neq 0$, so we centralize $A$ by subtracting its expectation $\E A=\pa{\frac{qs_{ij}}{N}}_{i,j=1}^{N}$ and normalize by $q^{-1/2}$. We denote the centered, normalized entries by $h_{ij}$.  Thus, $h_{ij}$ has law
\eq{\label{eq: 01 law}
	\P\qa{ h_{ij} = \frac{1}{\sqrt{q}}\pa{ 1-\frac{qs_{ij}}{N}}  } = \frac{qs_{ij}}{N}  \quad\text{and}\quad  \P[ h_{ij}=-\frac{1}{\sqrt{q}}  \frac{qs_{ij}}{N} ] = 1-\frac{qs_{ij}}{N} \,.
}

As before, using the Schur complement formula \eqref{eq: schur}, we find the equation $G_{ii}^{-1}=-z-\sum_k s_{ik} G_{kk} +\tilde{R}_i$, where the error term $\tilde{R}_i$ is defined as
\eq{
	\tilde{R}_i\deq h_{ii}+\sum_k s_{ik} \frac{G_{ik}G_{ki}}{G_{ii}}-Q_i \sum_{k,l}^{(i)}h_{ik}G_{kl}^{(i)}h_{li}-\frac{q}{N^2}\sum_k s_{ik}^2G_{kk} \,.
}
Thus, we get the same self-consistent equation for $u$, neglecting the error term.  We need to bound the additional error term $\frac{q}{N^2}\sum_k s_{ik}^2G_{kk} $.  We have, on the event $\phi \Lambda \prec N^{-c}$ (see Lemma \ref{lem: forbidden}), that
\eq{
	\tilde{R}_i \prec \theta_i \frac{q\theta_i}{N} \leq \ga_i \sqrt{\frac{q\theta_i^2}{N}} \sqrt{\frac{q}{N}}\leq \theta_i \sqrt{\frac{q}{N}}\,.
}
If $q\leq \sqrt{N}$, then $\tilde{R}_i \prec  \frac{\theta_i}{\sqrt{q}}\prec \theta_i \Phi$. Once we have this bound, together with large deviation bounds as in Lemma \ref{lem: LDE}, we readily get the local law, repeating the proof of Theorem \ref{thm: local law}.  In summary, we have Theorem \ref{thm: local law 3}.

\begin{theorem}\label{thm: local law 3}
Assume that $q=N^\kappa$ with $\kappa\leq 1/2$. Let $h_{ij}$ be defined as in \eqref{eq: 01 law}.  Assume that $I\subset \R$ is a bulk interval. Fix positive $\delta$ and define an $N$-dependent spectral domain $\cal{D}_\delta^I\deq \ha{ E+\ii\eta: E\in I, N^{\delta-1}\leq \eta \leq10}$, then on this domain $\Lambda\prec \Phi$.
\end{theorem}

As in the original model, the universality in the bulk holds for this model.  We state the universality theorem without proof, since the proof for Theorem \ref{thm: bulk universality} can be reproduced almost verbatim.

\begin{theorem}\label{thm: bulk universality 01}
Assume $q=N^\kappa$ with $\kappa \in(0,\frac{1}{2}]$.   Assume that $I\subset \R$ is a bulk interval.  Let $\rho^{(n)}$ be the n-point correlation functions of the eigenvalues of $H$ and $\rho^{(N)}$ be the density on $I$.  Let $O\in C^\infty_0(\R^n)$ be a test function. Fix a parameter $b=N^{c-1}$ for arbitrarily small $c$. Therefore, we have
\eq{\begin{split}
	\lim_{N\rightarrow \infty} \int_{E-b}^{E+b} \int_{\R^n} O(\alpha_1,\dots,\alpha_n) \left[ \frac{1}{\rho(E)^n}\rho^{(n)} \left( E'+\frac{\alpha_1}{N\rho(E)},\dots,E'+\frac{\alpha_n}{N\rho(E)}\right)\right.\\\left.
	-\frac{1}{(\rho_{sc}(E))^n} \rho_{GOE}^{(n)} \left( E''+\frac{\alpha_1}{\rho_{sc}^{(N)}(E)},\dots,E''+\frac{\alpha_n}{\rho_{sc}^{(N)}(E)}\right)\right]\dd \alpha_1\dots\dd\alpha_n \frac{\dd E'}{2b} =0
	\end{split}
}
for any $E''\in (-2,2)$.
\end{theorem}
\begin{remark}

With some more effort, we can remove the restriction $\kappa\leq 1/2$ in Theorem \ref{thm: local law 3} and Theorem \ref{thm: bulk universality 01}.  The idea is to solve a general self-consistent equation 
\eq{
	g_i=\frac{1}{-z-\sum_{k=1}^N \hat{s}_{ik}g_k}
}
for $i\in \N_{N}$, where $\hat{s}_{ik}=s_{ik}(1-\frac{qs_{ik}}{N})$.  We refrain from proving this in any detail here, as the most interesting case in applications is when $\kappa$ is small, that is, when the graph is very sparse.
\end{remark}

\appendix
\section{General self-consistent systems}\label{sec: general sce}
In this section, we forget about $N$ and consider the general self-consistent equation 
\be\label{sce}
	g_z(x)=-\frac{1}{z+(Sg_z)(x)}
\ee
for $ z\in \C^+$. Equation \eqref{sce} has earlier been studied in \cite{Anderson2008}, \cite{Girko} and \cite{Helton2007}.  Recently, it was extensively studied in \cite{Ajanki2015a}, and named quadratic vector equation (\textsc{qve}).   We shall prove a theorem about the uniqueness and existence for the solution.  Our theorem is contained in much more general results in  \cite{Ajanki2015a} and\cite{Helton2007}.  For the readers' convenience, we present here a different but short proof in a much simpler setting.

Denote $m(z)=\int_0^1 g(x)\dd x$. We let $x\in [0,1]$, $g_z \in L^2([0,1],\C^+)$ and $Sg(x):= \int_0^1 s(x,y)g(y)\dd y$ with $s\in L^2([0,1]^2)$, $s\geq 0$ on $[0,1]^2$.  

\begin{theorem}[Existence and uniqueness]\label{thm: solution}
For any fixed $z\in\C^+$, there uniquely exists a $L^2$ function $g_z:[0,1]\rightarrow \C^+$ that satisfies equation \eqref{sce}.  Moreover, $g$ as a function from $\C^+\rightarrow L^2[0,1]$ is holomorphic.
\end{theorem}

\begin{proof}
\textbf{Existence.}
Our strategy is, we first find a sequence of functions $g_k$ satisfying the equation at dyadic points $x=\frac{l}{2^k}, 0\leq l \leq 2^k$.  Second, we show that the sequence is bounded in $L^2[0,1]$.  Third, we use the fact that $S$ is a compact operator to find a convergent subsequence. Fourth, the limit turns out to be the solution.

Define a map $\phi_z:\overline{(\C^+)^{2^k}}\rightarrow(\C^+)^{2^k}$ as follows: For a vector $\zeta=\{\zeta^l\}_{l=1}^{2^k}\in(\C^+)^{2^k}$,
define a staircase function 
\be
	h_\zeta(x)=\sum_{l=1}^{2^k} 1_{[\frac{l-1}{2^k},\frac{l}{2^k}]}(x) \zeta^l\,.
\ee
Define a projection operator $P_k$ on $L^2[0,1]$ such that for $v\in L^2[0,1]$,
\be
	(P_kv)(x)=-\fint_{\frac{l-1}{2^k}}^{\frac{l}{2^k}} v(t)\dd t, x\in \bigg{(}\frac{l-1}{2^k},\frac{l}{2^k} \bigg{]}\,.
\ee
The projection $P_k$ is actually an operator from $ L^2\rightarrow \C^{2^k}$.  So we define
\be\label{phi}
	\phi_z(\zeta)=P_k \left(\frac{-1}{z+S h_\zeta}\right)
\ee
The function $\phi_z$ is continuous on $\phi_z:\overline{(\C^+)^{2^k}}\rightarrow(\C^+)^{2^k}$ and $\absa{\phi_z}_\infty$ is bounded by $\frac{1}{\eta}$. By Brouwer's fixed point theorem, there exists a $\zeta^*\in \C^+$ such that $\phi_z(\zeta^*)=\zeta^*$ and that $\absa{\zeta^*}_\infty\leq \frac{1}{\eta}$.

We set $g_k:=h_{\zeta^*}$.  By \eqref{phi}, we have $g_k=P_k  \left(\frac{-1}{z+S g_k}\right)$. The function $g_k$ is bounded by $\frac{1}{\eta}$, hence is bounded in $L^2[0,1]$. We can find a subsequence, still denoted by $g_k$, such that $Sg_k\rightarrow \mathfrak{g}$ in $L^2$.  Note that
\be
	\begin{split}
	\norma{g_k-\frac{-1}{z+\fg}}&\leq \norma{P_k\left(\frac{-1}{z+S g_k}-\frac{-1}{z+\fg}\right)+(P_k-I) \left(\frac{-1}{z+\fg}\right)}\\
						&\leq c \norma{g_k-g}+\norma{(P_k-I) \left(\frac{-1}{z+\fg}\right)}\rightarrow 0\,.
	\end{split}
\ee
Thus, $g_k$ converges to $g:=\frac{-1}{z+S\fg}$, which satisfies $g=-\frac{1}{z+Sg}$.

\textbf{Uniqueness.}
Fix $z=E+i\eta\in\C^+$.  Assume that there are two different solutions $g$ and $\hg$.  
Denote 
\be\label{a}
	a(x)=z+(Sg)(x)
\ee
and $\hat{a}$ similarly.  Since $\im g(x)>\frac{(S\im g)(x)}{\absa{a(x)}^2}$, we have
\be\label{a7}
	\absa{a(x)}\im g(x) >\int_0^1 \frac{s(x,y)}{\absa{a(x)a(y)}} \absa{a(y)}\im g(y) \dd y\,.
\ee
Next, we use a generalization of Perron-Frobinius theorem.

\begin{theorem}[The Krein-Rutman theorem]
Let $X$ be a Banach space, $K\subset X$ a total cone (that is, $\overline{K-K}=X$) and $T:X\rightarrow X$ a compact linear operator that is positive (that is, $T(K)\subset K$) with positive spectral radius $r(T)$, then $r(T)$ is an eigenvalue of $T$ with an eigenvector $u\in K\backslash \{0\}:Tu=r(T)u$. 
\end{theorem}
In our case, $X=L^2[0,1]$, $K=\{f\in X, f\geq 0\}$, and $T$ is the integral operator with kernel $\frac{s(x,y)}{\absa{a(x)a(y)}}$. The operator $T$ is compact because its kernel is continuous.  The theorem tells us that there is a $u(x)\in K\backslash \{0\}$ such that 
\eq{
	\int_0^1\frac{s(x,y)}{\absa{a(x)a(y)}}u(y)\dd y =r(T) u(x)\,.
}
We take the inner product  of \eqref{a7} with the eigenvalue $u$ to get $\langle u, a\im g\rangle  > r(T)\langle u, a\im g\rangle$. Thus, $r(T)<1$.  Since $T$ is symmetric, this means $\norma{T}<1$.  Similarly $\norma{\hat{T}}<1$.

On the other hand
\be
	\begin{split}
	\absa{g(x)-\hg(x)} &=\frac{\int s(x,y)\absa{g(y)-\hg(y)} \dd y }{a(x)\hat{a}(x)}
	\end{split}
\ee		
Denote $w(x)=\sqrt{a(x)\hat{a}(x)}\absa{g(x)-\hg(x)}$,
\be
	w(x)\leq \int _0^1\frac{s(x,y)}{\sqrt{a(x)\hat{a}(x)}\sqrt{a(y)\hat{a}(y)}}w(y)\dd y\\\leq \left(\frac{1}{2}(T+\hat{T}) w \right)(x)\,.
\ee
So $\norma{w}\leq \norma{\frac{1}{2}(T+\hat{T})}\norma{w} <\norma{w}$, a contradiction.

\end{proof}

The following theorems concern the stability of the equation.   Here we assume $s$ to be bounded below, so that when $\eta$ gets small, we still have a bound on $g_i$.  We omit the proof here, because the proof is essentially a reinstatement of the uniqueness of solutions.  Our theorem are contained in more general theorems in \cite{Ajanki2015a}.  We present our theorem here for the readers' convenience, since they are easier to read.
\begin{theorem}[Stability in the bulk]\label{bulk stability}
Assume that $s$ is bounded below on $[0,1]^2$.  Assume $g$ is the solution to \eqref{sce}. Denote $m(z)=\int_0^1 g(x)\dd x$. Let $I$ be a bounded interval such that $\im m(z)$ is bounded below on $\mathcal{D}:=\{E+\ii\eta: E \in I, 0<\eta\leq 10\}$.  

Thus, there are positive constants $C$ and $\e_0$ such that for any $z\in \mathcal{D}$, if $\hg$ satisfies
\be
	\norma{\hg+\frac{1}{z+S\hg}}\leq \e\leq \e_0, \norma{\hg-g_z}\leq \e_0\,,
\ee
then on $\mathcal{D}$ we have $\norma{\hg-g_z}\leq C\e$.
\end{theorem}
\begin{theorem}[Stability in $S$]\label{s stability}
Assume that $s$ and $\hat{s}$ are bounded below on $[0,1]^2$.  Assume $g$ is the solution to \eqref{sce} and $\hg$ solves \eqref{sce} with $\hS$ replacing $S$.  Let $I$ be a bounded interval such that $\im m$ is bounded below on $\mathcal{D}:=\{E+\ii\eta:E \in I, 0<\eta \leq 10\}$.  Thus, there are positive $C$ and $\e_0$ such that for any $z\in \mathcal{D}$, if $\hS$ satisfies $\norma{S-\hS}\leq \e\leq \e_0$, then on $\mathcal{D}$, we have $\norma{\hg-g}\leq C\e$.
\end{theorem}

\section{Local law and bulk universality of a general ensemble}\label{sec: general local law}
In this section we consider $s_{ij}\in[c/N,C/N]$ for some $c,C>0$, and do not put low-rank conditions on $(s_{ij})$.  We have Theorem \ref{weak law sparse} and \ref{bulk universality general} which are parallel to Theorem \ref{thm: local law} and Theorem \ref{thm: bulk universality}.   Again, we point out that in the special case $q=N$, Theorem \ref{weak law sparse} and \ref{bulk universality general} are contained in Theorem 1.6 and Theorem 1.15 in \cite{Ajanki2015}. 

We set a sparse parameter $q=N^\kappa$ with $\kappa \in (0,1)$. Assume either the biased case where 
\eq{
	\P\qa{ h_{ij}=\frac{1}{\sqrt{q}} }=qs_{ij} \quad\text{and}\quad \P\qa{h_{ij}=0}=1-qs_{ij}
}
after centralization or the unbiased one where
\eq{
	\P\qa{ h_{ij}=\pm\frac{1}{\sqrt{q}} }=\frac{qs_{ij}}{2} \quad\text{and}\quad \P\qa{h_{ij}=0}=1-qs_{ij}\,.
}
In either case, we have moment bounds
\be
	\E\absa{h_{ij}}^p\leq \frac{C^p}{Nq^{\frac{p}{2}-1}}\,.
\ee
So, we have large deviation bounds, assuming $\absa{A_j}\prec 1$ and $\absa{B_{kl}}\prec 1$:
\be
	\sum_{j}(h_{ij}^2-s_{ij}) A_j \prec  \frac{1}{\sqrt{q}}+\left(\frac{1}{N}\sum_j \absa{A_j}^2\right)^{1/2}
\ee
and
\be
	\sum_{k\neq l} h_{ik} B_{kl} h_{lj}\prec \frac{1}{\sqrt{q}}+\left(\frac{1}{N}\sum_{k,l} \absa{B_{kl}}^2 \right)^{1/2}\,.
\ee
As before, we set a control parameter $\Phi\deq\frac{1}{\sqrt{q}}+\frac{1}{\sqrt{N\eta}}$. We think of $s^{(N)}_{ij}$ as a step function $s^{(N)}$ on $[0,1]^2$, that is, $s^{(N)}(x,y) = Ns^{(N)}_{ij}$, for $x\in\left(\frac{i-1}{N}, \frac{i}{N}\right]$ and $y\in\left(\frac{j-1}{N}, \frac{j}{N}\right]$. Let $S^{(N)}$ be the integral operator with kernel $s^{(N)}$. Thus, $S^{(N)}$ is uniformly bounded in the Hilbert-Shmidt norm. Consider the self-consistent equation
\be\label{sce n}
	g^{(N)}(x)=-\frac{1}{z+S^{(N)}g^{(N)}(x)}\,,
\ee
where $g: \C^+\times [0,1] \rightarrow \C^+$. By Theorem \ref{thm: solution}, $g^{(N)}$ exists and is unique.  It is easy to see that $g^{(N)}_z$ is also a step function on $[0,1]$, whose value on $(\frac{i-1}{N},\frac{i}{N}]$ we denote by $g_i$. Define $m^{(N)}\deq\int_0^1 g^{(N)}\dd x$. In the rest of this paper, we will omit the $N$ for simplicity.  As before, we are going to analyze the Green's function defined by $G\deq (H-z)^{-1}$. First, we apply Schur's complement formula \eqref{eq: schur}, to get $G_{ii}^{-1}=-z-\sum_k s_{ik} G_{kk} +R_i$, where the error term $R_i$ is defined as
\be	\label{eq: error B}
	R_i \deq h_{ii}+\sum_k s_{ik} \frac{G_{ik}G_{ki}}{G_{ii}}-Q_i \sum_{k,l}^{(i)}h_{ik}G_{kl}^{(i)}h_{li}\,.
\ee
Neglecting the error term $R_i$, we get equation \eqref{sce n}.

We define bulk interval as in  Definition \ref{def: bulk interval}. We have a local law as follows.   
\begin{theorem}[Local law for the sparse model]\label{weak law sparse}
Let $H$ be a family of random matrices defined above. Assume that $s_{ij}\in[\frac{c}{N},\frac{C}{N}]$ for some positive constants $c$ and $C$, that $g=g^{(N)}$ solves the self-consistent equation \eqref{sce n}, and that $m(z)=m^{(N)}(z)=\int_0^1 g^{(N)}(x)\dd x$.  Assume that $I$ is a bulk interval. Thus, on the domain 
\eq{
	\cal{D}_\delta^I\deq \ha{ E+\ii\eta: E\in I, N^{\delta-1}\leq \eta \leq10}\,.
}
we have $\Lambda\prec \Phi$.
\end{theorem}
\begin{proof}[Sketch of proof]
Assume that $\phi\Lambda\prec N^{-c}$ where $\phi$ is some event.  We shall show $\La\prec\Phi :=\frac{1}{\sqrt{N\eta}}$.  In the following estimates, we omit universal constants that are independent of $N$.
First, we get some information from the fact that $\im m $ is bounded below, that is,
\be
	\absa{g_i}\leq \frac{1}{\sum_j s_{ij}\im g_j}\leq \frac{1}{\im m}\leq \frac{1}{c_I}\,.
\ee
So $g$ is uniformly bounded in $L^\infty[0,1]$. Moreover,
\be
	\im g_i \geq \frac{\sum_j s_{ij} \im g_j }{\absa{\absa{z}+\sum_j s_{ij} \frac{1}{c_I}}^2}\geq \im m \geq c_I\,.
\ee
So $ \im g_i$ is uniformly bounded below.  Therefore $\absa{G_{ii}}\asymp 1$.  Moreover, $\absa{G_{ii}^{\T}}\prec 1$ by resolvent identities.

For $\Lambda_o$, we have $G_{ij}\prec \Phi$. For $\Lambda_d$, we firstly have $R_i\prec \Phi$, which implies
\be
	\max_{i}\absa{{G_{ii}+\frac{1}{z+\sum_j s_{ij} G_{jj}}}}\prec \Phi\,.
\ee
Theorem \ref{bulk stability} tells us that $\sqrt{\frac{1}{N}\sum_i\absa{G_{ii}-g_i}^2}\prec \Phi$. By the $L^\infty$-boundedness of $s$, we have
\be
	\absa{\frac{1}{z+\sum_j s_{ij} G_{jj}}-\frac{1}{z+(Sg)_i}}\prec \Phi\,.
\ee
Therefore,
\be
	\Lambda_d \leq \max_{i}\absa{{G_{ii}+\frac{1}{z+\sum_j s_{ij} G_{jj}}}}+\max_i\absa{\frac{1}{z+\sum_j s_{ij} G_{jj}}-\frac{1}{z+(Sg)_i}}\prec \Phi\,.
\ee
Finally, we apply a continuity argument to conclude the proof.
\end{proof}

Note that the theorem does not assume any limit of $s^{(N)}$. Instead, it requires $m^{(N)}$ to have some common ``bulk,'' that is, a uniform lower bound for $\im m$ near some interval $I\subset\R$.  In particular, if $s^{(N)}$ is close enough to some certain $s^*$, whose $\im m^* $ is bounded below on $\cal{D}_0^I$, then we have uniform lower bound for $\im m$.

\begin{corollary}
Assume the same conditions as in Theorem \ref{weak law sparse}.  We assume $s^*$ satisfies the conditions of Lemma \ref{bulk stability}, and that $\im m^*(z)$ is bounded below on some spectral domain ${\mathcal{D}_0^I\deq \{E+i\eta: E \in I, \eta\in (0,10]\}}$. For any positive $\de$, define ${\mathcal{D}_\delta^I\deq \{E+i\eta: E \in I, \eta\in (N^{\de-1},10]\}}$. Thus, there exists some positive $\e_0$ such that if $\norma{s^{(N)}-s^*} \leq \e_0$, we have $\Lambda\prec \Phi$.
\end{corollary}

\begin{theorem}\label{bulk universality general}
Let $H^{(N)}$ be a family of random matrices defined above. Assume that $s_{ij}\in[\frac{c}{N},\frac{C}{N}]$ for some positive constants $c$ and $C$, that $g=g^{(N)}$ solves Equation \eqref{sce n}, and that ${m(z)=m^{(N)}(z)=\int_0^1 g^{(N)}(x)\dd x}$.  Let $I$ be a bulk interval. Let $\rho^{(n)}$ be the n-point correlation functions of the eigenvalues of $H$ and $\rho^{(N)}$ be the density on $I$.  Let $O\in C^\infty_0(\R^n)$ be a test function. Fix a parameter $b=N^{c-1}$ for arbitrarily small $c$. We have,
\eq{\begin{split}
	\lim_{N\rightarrow \infty} \int_{E-b}^{E+b} \int_{\R^n} O(\alpha_1,\dots,\alpha_n) \left[ \frac{1}{\rho(E)^n}\rho^{(n)} \left( E'+\frac{\alpha_1}{N\rho(E)},\dots,E'+\frac{\alpha_n}{N\rho(E)}\right)\right.\\\left.
	-\frac{1}{(\rho_{sc}(E))^n} \rho_{GOE}^{(n)} \left( E''+\frac{\alpha_1}{\rho_{sc}^{(N)}(E)},\dots,E''+\frac{\alpha_n}{\rho_{sc}^{(N)}(E)}\right)\right]\dd \alpha_1\dots\dd\alpha_n \frac{\dd E'}{2b} =0
	\end{split}
}
for any $E''\in (-2,2)$.
\end{theorem}

The proof is parallel to that of Theorem \ref{thm: bulk universality}.

\bibliographystyle{abbrv}
\bibliography{random_graph.bib}

\begin{thebibliography}{10}

\bibitem{Adlam2014}
B.~Adlam and M.~A. Nowak.
\newblock {Universality of fixation probabilities in randomly structured
  populations}.
\newblock {\em Sci. Rep.}, 4, July 2014.

\bibitem{Aiello2001}
W.~Aiello, F.~Chung, and L.~Lu.
\newblock {A random graph model for power law graphs}.
\newblock {\em Experimental Mathematics}, 10(1):53--66, 2001.

\bibitem{Ajanki2015a}
O.~Ajanki, L.~Erd\H{o}s, and T.~Kruger.
\newblock Quadratic vector equations on complex upper half-plane.
\newblock {\em arXiv.org}, (arXiv:1506.05095), June 2015.

\bibitem{Ajanki2015}
O.~Ajanki, L.~Erdos, and T.~Kruger.
\newblock {Universality for general Wigner-type matrices}.
\newblock {\em arXiv.org}, (arXiv:1506.05098), 2015.

\bibitem{Albert1999}
R.~Albert, H.~Jeong, and A.~L. Barabasi.
\newblock {Internet: Diameter of the World-Wide Web}.
\newblock {\em Nature}, 401(6749):130--131, 1999.

\bibitem{Aldous2002}
D.~Aldous and J.~Fill.
\newblock {Reversible Markov chains and random walks on graphs}, 2002.

\bibitem{Anderson2010}
G.~W. Anderson, A.~Guionnet, and O.~Zeitouni.
\newblock {\em {An introduction to random matrices}}, volume 118 of {\em
  Cambridge studies in advanced mathematics}.
\newblock Cambridge University Press, Cambridge ;New York, 2010.

\bibitem{Anderson2008}
G.~W. Anderson and O.~Zeitouni.
\newblock A law of large numbers for finite-range dependent random matrices.
\newblock {\em Communications in Pure and Applied Mathematics},
  (61(8):1118-1154), 2008.

\bibitem{Antal2006}
T.~Antal, S.~Redner, and V.~Sood.
\newblock {Evolutionary dynamics on degree-heterogeneous graphs}.
\newblock {\em Physics Review Letter}, 96:188104, 2006.

\bibitem{Barabasi1999}
A.-L. Barab\'{a}si and R.~Albert.
\newblock {Emergence of scaling in random networks}.
\newblock {\em Science}, 286(5439):509--512, Oct. 1999.

\bibitem{Bollobas1998}
B.~Bollob\'{a}s.
\newblock {\em {Random graphs}}.
\newblock Springer, 1998.

\bibitem{Chakrabarti1999}
S.~Chakrabarti, B.~Dom, D.~Gibson, J.~M. Kleinberg, S.~R. Kumar, P.~Raghavan,
  S.~Rajagopalan, and A.~Tomkins.
\newblock {Hypersearching the Web}.
\newblock {\em Scientific American}, 280(6):54--60, 1999.

\bibitem{Chung2002}
F.~Chung and L.~Lu.
\newblock {The average distances in random graphs with given expected degrees}.
\newblock {\em Proceedings of the National Academy of Sciences},
  99(25):15879--15882, 2002.

\bibitem{Chung2006a}
F.~Chung and L.~Lu.
\newblock {\em {Complex Graphs and Networks}}, volume 107.
\newblock 2006.

\bibitem{Chung2003}
F.~Chung, L.~Lu, and V.~Vu.
\newblock {Spectra of random graphs with given expected degrees.}
\newblock {\em Proceedings of the National Academy of Sciences of the United
  States of America}, 100(11):6313--6318, 2003.

\bibitem{Dembo2014}
A.~Dembo, A.~Montanari, A.~Sly, and N.~Sun.
\newblock {The replica symmetric solution for Potts models on d-regular
  graphs}.
\newblock {\em Communications in Mathematical Physics}, 327(2):551--575, 2014.

\bibitem{Erdos2012a}
L.~Erd\H{o}s, A.~Knowles, H.~T. Yau, and J.~Yin.
\newblock {Spectral Statistics of Erd\H{o}s-R\'{e}nyi Graphs II: Eigenvalue
  Spacing and the Extreme Eigenvalues}.
\newblock {\em Communications in Mathematical Physics}, 314(3):587--640, 2012.

\bibitem{Erdos2013b}
L.~Erd\H{o}s, A.~Knowles, H.~T. Yau, and J.~Yin.
\newblock {Spectral statistics of Erdos-R\'{e}nyi graphs i: Local semicircle
  law}.
\newblock {\em Annals of Probability}, 41(3 B):2279--2375, 2013.

\bibitem{Erdos2013a}
L.~Erd\H{o}s, A.~Knowles, H.-T. Yau, and J.~Yin.
\newblock {The local semicircle law for a general class of random matrices}.
\newblock {\em Electron. J. Probab}, 18(59):1--58, 2013.

\bibitem{Erdos2009b}
L.~Erd\H{o}s, S.~Peche, J.~A. Ramirez, B.~Schlein, and H.-T. Yau.
\newblock {Bulk Universality for Wigner Matrices}.
\newblock page~23, May 2009.

\bibitem{Erdos2011}
L.~Erd\H{o}s, B.~Schlein, and H.-T. Yau.
\newblock {Universality of random matrices and local relaxation flow}.
\newblock {\em Invent. Math.}, 185(1):75--119, 2011.

\bibitem{Erdos2012c}
L.~Erd\H{o}s and H.-T. Yau.
\newblock {Gap Universality of Generalized Wigner and beta-Ensembles}.
\newblock {\em arXiv Prepr. arXiv1211.3786}, pages 1--83, 2012.

\bibitem{Erdos2012b}
L.~Erd\H{o}s, H.-T. Yau, and J.~Yin.
\newblock {Rigidity of eigenvalues of generalized Wigner matrices}.
\newblock {\em Adv. Math. (N. Y).}, 229(3):1435--1515, 2012.

\bibitem{Erdos1959}
P.~Erd\H{o}s and A.~R\'{e}nyi.
\newblock {On random graphs}.
\newblock {\em Publicationes Mathematicae}, 6:290--297, 1959.

\bibitem{Erdos1960}
P.~Erd{\H o}s and A.~R\'{e}nyi.
\newblock {On the evolution of random graphs}.
\newblock {\em Publications of the Mathematical Institute of the Hungarian
  Academy of Sciences}, 5, 1960.

\bibitem{Girko}
V.~L. Girko.
\newblock {\em Theory of stochastic canonical equations. Vol. I}, volume 535 of
  {\em Mathematics and its Applications}.
\newblock Kluwer Academic Publishers, Dordrecht, 2001.

\bibitem{Goffman1969}
C.~Goffman.
\newblock {And what is your Erdos number?}
\newblock {\em American Mathematical Monthly}, page 791, 1969.

\bibitem{Goldreich2008}
O.~Goldreich.
\newblock {Computational complexity: a conceptual perspective}.
\newblock {\em ACM SIGACT News}, 39(3):35--39, 2008.

\bibitem{Helton2007}
J.~W. Helton, R.~R. Far, and R.~Speicher.
\newblock Operator-valued semicircular elements: Solving a quadratic matrix
  equation with positivity constraints.
\newblock {\em International Mathematics Research}, 2007(rnm 086), 2007.

\bibitem{Landon2015}
B.~Landon, J.~Huang, and H.-T. Yau.
\newblock {Bulk universality of sparse random matrices}.
\newblock (arXiv:1504.05170):1--20, 2015.

\bibitem{Landon2015a}
B.~Landon and H.-T. Yau.
\newblock Convergence of local statistics of dyson brownian motion.
\newblock {\em arXiv.org}, (arXiv:1504.03605v1), 2015.

\bibitem{Lieberman2005}
E.~Lieberman, C.~Hauert, and M.~A. Nowak.
\newblock {Evolutionary dynamics on graphs}.
\newblock {\em Nature}, (JANUARY):312--316, 2005.

\bibitem{Ohtsuki2006}
H.~Ohtsuki, C.~Hauert, E.~Lieberman, and M.~A. Nowak.
\newblock {A simple rule for the evolution of cooperation on graphs and social
  networks}.
\newblock {\em Nature}, 441:502--505, May 2006.

\bibitem{Rodgers2006}
G.~J. Rodgers, K.~Austin, B.~Kahng, and D.~Kim.
\newblock {Eigenvalue spectra of complex networks}.
\newblock 9431, 2006.

\bibitem{Tao2010}
T.~Tao and V.~Vu.
\newblock {Random matrices: Universality of local eigenvalue statistics up to
  the edge}.
\newblock {\em Commun. Math. Phys.}, 298(2):549--572, 2010.

\bibitem{Tao2011a}
T.~Tao and V.~Vu.
\newblock {Random matrices: Universality of local eigenvalue statistics}.
\newblock {\em Acta Math.}, 206(1):127--204, 2011.

\bibitem{Vadhan2012}
S.~Vadhan.
\newblock {Pseudorandomness. Foundations and Trends in Theoretical Computer
  Science}.
\newblock 2012.

\bibitem{Wasserman1994}
S.~Wasserman and K.~Faust.
\newblock {\em {Social network analysis : methods and applications}},
  volume~24.
\newblock 1994.

\bibitem{Watts1998}
D.~J. Watts and S.~H. Strogatz.
\newblock {Collective dynamics of 'small-world' networks.}
\newblock {\em Nature}, 393:440--442, June 1998.

\end{thebibliography}

\end{document}